\documentclass[11pt]{amsart}
\usepackage[leqno]{amsmath}
\usepackage{amsfonts, amssymb,srcltx, amsopn, color}

\usepackage{graphicx}
\usepackage{xspace}

\usepackage{epsfig}
\usepackage{color}
\usepackage{amsmath}
\usepackage{amssymb}

\newtheorem{lemma}{Lemma}
\newtheorem{proposition}{Proposition}
\newtheorem{theorem}{Theorem}
\newtheorem{corollary}{Corollary}

\theoremstyle{definition}

\newtheorem{remark}{Remark}

\usepackage{color}
\newcommand{\commentout}[1]{}

\newcommand{\cB}{\ensuremath{\mathcal{B}} \xspace}
\newcommand{\N}{\ensuremath{\mathbb{N}} \xspace}
\usepackage{graphicx}
\usepackage[algoruled,vlined,english,norelsize]{algorithm2e} 

\begin{document}

\thispagestyle{empty}
\centerline{\Large\bf Cop and robber game and hyperbolicity}

\vspace{10mm}

\centerline{{\sc J\'er\'emie Chalopin}$^{1}$, {\sc Victor
Chepoi}$^{1}$, {\sc Panos Papasoglu}$^{2}$, and {\sc Timoth\'ee Pecatte}$^{3}$}

\vspace{3mm}

\date{\today}

\medskip
\begin{small}
\centerline{$^{1}$Laboratoire d'Informatique Fondamentale, Aix-Marseille Universit\'e and CNRS,}
\centerline{Facult\'e des Sciences de Luminy, F-13288 Marseille Cedex 9, France}

\centerline{\texttt{\{jeremie.chalopin, victor.chepoi\}@lif.univ-mrs.fr}}

\medskip
\centerline{$^2$Mathematical Institute}
\centerline {24--29 St Giles, Oxford, OX1 3LB, England}
\centerline{\texttt{ papazoglou@maths.ox.ac.uk}}

\medskip
\centerline{$^{3}$Ecole Normale Sup\'erieure de Lyon,}
\centerline{46, all\'ee d'Italie, 69342 Lyon Cedex 07, France}
\centerline{\texttt{timothee.pecatte@ens-lyon.fr}}

\end{small}

\bigskip\bigskip\noindent {\footnotesize {\bf Abstract.}  In this
  note, we prove that all cop-win graphs $G$ in the game in which the
  robber and the cop move at different speeds $s$ and $s'$ with
  $s'<s$, are $\delta$-hyperbolic with $\delta=O(s^2)$. We also show
  that the dependency between $\delta$ and $s$ is linear if
  $s-s'=\Omega(s)$ and $G$ obeys a slightly stronger condition.  This
  solves an open question from the paper \emph{J. Chalopin et al., Cop
    and robber games when the robber can hide and ride, SIAM
    J. Discr. Math. 25 (2011) 333--359}. Since any $\delta$-hyperbolic
  graph is cop-win for $s=2r$ and $s'=r+2\delta$ for any $r>0$, this
  establishes a new --game-theoretical-- characterization of Gromov
  hyperbolicity. We also show that for weakly modular graphs the
  dependency between $\delta$ and $s$ is linear for any $s'<s$. Using
  these results, we describe a simple constant-factor approximation of
  the hyperbolicity $\delta$ of a graph on $n$ vertices in $O(n^2)$
  time when the graph is given by its distance-matrix.}

\section{Introduction}

The cop and robber game originated in the 1980's with the work of
Nowakowski, Winkler \cite{NowWin}, Quilliot \cite{Qui83}, and Aigner,
Fromme \cite{AigFro}, and since then has been intensively investigated
by many authors under numerous versions and generalizations.
Cop and robber is a  pursuit-evasion game played on finite undirected
graphs $G=(V,E)$. Player cop $\mathcal C$ attempts to capture the robber $\mathcal R$.
At the beginning of the game, $\mathcal C$ chooses a vertex of $G$,
then $\mathcal R$ chooses another vertex.
Thereafter, the two sides
move alternatively, starting with $\mathcal C,$ where a move is to
slide along an edge of $G$ or to stay at the same vertex.
The objective of $\mathcal C$ is to capture $\mathcal R$,
i.e., to be at some moment in time at the same
vertex as the robber.  The objective of $\mathcal R$ is to continue
evading the cop.  A {\it cop-win graph}
\cite{AigFro,NowWin,Qui83} is a graph in which $\mathcal C$
captures $\mathcal R$  after a finite number of moves from any possible
initial positions of $\mathcal C$ and $\mathcal R.$

In this paper, we investigate a natural extension of the cop and robber game
in which the cop $\mathcal C$ and the robber $\mathcal R$ move at
speeds $s'\ge 1$ and $s\ge 1,$ respectively.  This game was introduced
and thoroughly investigated in \cite{ChChNiVa}.  It generalizes the cop and fast
robber game from \cite{FoGoKrNiSu} and can be viewed as the discrete
version of some pursuit-evasion games played in continuous
domains \cite{FoTi}.  The unique difference of this ``$(s,s')$-cop and
robber game'' and the classical cop and robber game is that at each step,
$\mathcal C$ can move along a path of length at most $s'$ and  $\mathcal R$ can
move along a path of length at most $s$ not traversing the position
occupied by the cop. Following \cite{ChChNiVa}, we will denote the class of cop-win
graphs for this game by ${\mathcal C}{\mathcal W}{\mathcal F}{\mathcal R}(s,s')$.

Analogously to the characterization of classical cop-win graphs given in
\cite{NowWin,Qui83}, the $(s,s')$-cop-win graphs have been characterized
in \cite{ChChNiVa} via a special dismantling scheme. It was also shown in \cite{ChChNiVa} that any
$\delta$-hyperbolic graph in the sense of Gromov \cite{Gr} belongs to the
class ${\mathcal C}{\mathcal W}{\mathcal F}{\mathcal R}(2r,r+2\delta)$ for any $r>0$  and
that, for any $s\geq 2s'$, the graphs in ${\mathcal C}{\mathcal W}{\mathcal F}{\mathcal R}(s,s')$ are
$(s-1)$-hyperbolic. Finally, \cite{ChChNiVa} conjectures that all graphs
of ${\mathcal C}{\mathcal W}{\mathcal F}{\mathcal R}(s,s')$ with $s'<s,$ are $\delta$-hyperbolic,
where $\delta$ depends only of $s$ and establishes this conjecture for  Helly graphs and bridged
graphs, two important classes of weakly modular graphs.

In this note, we confirm the conjecture of \cite{ChChNiVa} by showing
that if $s'<s,$ then any graph of ${\mathcal C}{\mathcal W}{\mathcal
  F}{\mathcal R}(s,s')$ is $\delta$-hyperbolic with $\delta=O(s^2).$
The proof uses the dismantling characterization of $(s,s')$-cop-win
graphs and the characterization of $\delta$-hyperbolicity via the
linear isoperimetric inequality. We show that the dependency between
$\delta$ and $s$ is linear if $s-s'=\Omega(s)$ and $G$ satisfies a
slightly stronger dismantling condition. We also show that weakly
modular graphs from ${\mathcal C}{\mathcal W}{\mathcal F}{\mathcal
  R}(s,s')$ with $s'<s$ are $184s$-hyperbolic.  All this allows us to
approximate within a constant factor the least value of $\delta$ for
which a finite graph $G=(V,E)$ is $\delta$-hyperbolic in
$O(|V|^2)$ time once the distance-matrix of $G$ has been
computed.

\section{Preliminaries}

\subsection{Graphs}
All graphs $G=(V,E)$ occurring in this paper are undirected,
connected, without loops or multiple edges, but not
necessarily finite or locally-finite.  For a subset $A\subseteq V,$ the subgraph of $G=(V,E)$  {\it induced by} $A$
is the graph $G(A)=(A,E')$ such that $uv\in E'$ if and only if $u,v\in A$ and $uv\in E$. We
will write $G-\{ x\}$ instead of $G(V\setminus\{ x\})$.
The {\it distance} $d(u,v):=d_G(u,v)$
between two vertices $u$ and $v$ of  $G$ is the length (number of
edges) of a $(u,v)$-{\it geodesic}, i.e., a shortest $(u,v)$-path. For a vertex $v$ of $G$ and an
integer $r\ge 1$, we will denote  by $B_r(v,G)$ the \emph{ball} in $G$
of radius $r$ centered at  $v$, i.e.,
$B_r(v,G)=\{ x\in V: d(v,x)\le r\}.$ (We will write $B_r(v)$ instead of
$B_r(v,G)$ when this is clear from the context).
Let  $B_r(x,G - \{ y\})$ be the ball of radius $r$ centered
at $x$ in the graph $G -\{ y\}.$ The {\it interval}
$I(u,v)$ between $u$ and $v$ consists of all vertices on
$(u,v)$-geodesics, that is, of all vertices (metrically) {\it between} $u$
and $v$: $I(u,v)=\{ x\in V: d(u,x)+d(x,v)=d(u,v)\}.$

Three vertices $v_1,v_2,v_3$ of a graph $G$ form a {\it metric triangle} $v_1v_2v_3$ if the intervals $I(v_1,v_2), I(v_2,v_3),$ and
$I(v_3,v_1)$ pairwise intersect only in the common end-vertices. If
$d(v_1,v_2)=d(v_2,v_3)=d(v_3,v_1)=k,$ then this metric triangle is
called {\it equilateral} of {\it size} $k.$ A metric triangle
$v_1v_2v_3$ of $G$ is a {\it quasi--median} of the triplet $x,y,z$
if the following metric equalities are satisfied:
$$\begin{array}{l}
d(x,y)=d(x,v_1)+d(v_1,v_2)+d(v_2,y),\\
d(y,z)=d(y,v_2)+d(v_2,v_3)+d(v_3,z),\\
d(z,x)=d(z,v_3)+d(v_3,v_1)+d(v_1,x).\\
\end{array}$$
Every triplet $x,y,z$ of a graph has at least one quasi-median:
first select any vertex $v_1$ from $I(x,y)\cap I(x,z)$ at maximal
distance to $x,$ then select a vertex $v_2$ from $I(y,v_1)\cap
I(y,z)$ at maximal distance to $y,$ and finally select any vertex
$v_3$ from $I(z,v_1)\cap I(z,v_2)$ at maximal distance to $z.$


\subsection{$\delta$-Hyperbolicity}


A metric space $(X,d)$ is $\delta$-{\it hyperbolic}
\cite{AlBrCoFeLuMiShSh,BrHa,Gr} if for any four points $u,v,x,y$ of
$X$, the two larger of the three distance sums $d(u,v)+d(x,y)$,
$d(u,x)+d(v,y)$, $d(u,y)+d(v,x)$ differ by at most $2\delta \geq 0$. A
graph $G = (V,E)$ is $\delta$-\emph{hyperbolic} if $(V,d_G)$ is
$\delta$-{hyperbolic}.  In case of geodesic metric spaces and graphs,
$\delta$-hyperbolicity can be defined in several other equivalent
ways. Here we recall some of them, which we will use in our proofs.

Let $(X,d)$ be a metric space.  A {\it geodesic segment} joining two
points $x$ and $y$ from $X$ is a map $\rho$ from the segment $[a,b]$
of ${\mathbb R}^1$ of length $|a-b|=d(x,y)$ to $X$ such that
$\rho(a)=x, \rho(b)=y,$ and $d(\rho(s),\rho(t))=|s-t|$ for all $s,t\in
[a,b].$ A metric space $(X,d)$ is {\it geodesic} if every pair of
points in $X$ can be joined by a geodesic segment. Every (combinatorial)
graph $G=(V,E)$ equipped with its standard
distance $d:=d_G$ can be transformed into a geodesic (network-like)
space $(X_G,d)$ by replacing every edge $e=(u,v)$ by a segment
$\gamma_{uv}=[u,v]$ of length 1; the segments may intersect only at
common ends.  Then $(V,d_G)$ is isometrically embedded in a natural
way in $(X_G,d)$. $X_G$ is often called a {\it metric graph}.  The
restrictions of geodesics of $X_G$ to the set of vertices $V$ of $G$ 
are the shortest paths of $G$.
For simplicity of notation and brevity (and if not said otherwise), in all
subsequent results, by a geodesic $[x,y]$ in a graph $G$ we will mean an arbitrary shortest
path between two vertices $x,y$ of $G$.

Let $(X,d)$ be a geodesic space.  A \textit{geodesic triangle}
$\Delta(x,y,z)$ with $x, y, z \in X$ is the union $[x,y] \cup [x,z]
\cup [y,z]$ of three geodesic segments connecting these vertices.  A
geodesic triangle $\Delta(x,y,z)$ is called $\delta$-{\it slim} if for
any point $u$ on the side $[x,y]$ the distance from $u$ to $[x,z]\cup
[z,y]$ is at most $\delta$. For graphs, we ``discretize'' this notion in the following way.
We say that the geodesic triangles of a
graph $G$ are $\delta$-\emph{slim} if for any triplet $x,y,z$ of vertices of $G$, for
any (graph) geodesics  $[x,y], [x,z], [y,z],$ and for any vertex $u\in [x,y]$,
there exists $v\in [x,z] \cup [y,z]$ such that $d(u,v)\leq \delta$.

Note that if the metric graph $(X_G,d)$ is $\delta$-hyperbolic (resp.,
has $\delta$-slim geodesic triangles) as a geodesic metric space, then
the combinatorial graph $G$ is $\delta$-hyperbolic (resp., has
$\delta$-slim geodesic triangles). Conversely, if $G$ is $\delta$-hyperbolic
(resp., has $\delta$-slim geodesic triangles), then $(X_G,d)$ is
$(\delta+2)$-hyperbolic (resp., has $(\delta + \frac{1}{2})$-slim geodesic
triangles).


The following result shows that hyperbolicity of a geodesic space is
equivalent to having slim geodesic triangles (the same result holds
for graphs).

\begin{proposition} \label{hyp_charact}  \label{prop-relations-delta}\cite{AlBrCoFeLuMiShSh,BrHa,Gr}
If all geodesic triangles of a geodesic metric space (X,d) are $\delta$-slim, then $X$ is
$8\delta$-hyperbolic. Conversely, if a geodesic space $(X,d)$ is
$\delta$-hyperbolic, then all its geodesic triangles are $3\delta$-slim.
\end{proposition}




More recently, Soto~\cite{Soto-PhD} proved a sharp bound on the
hyperbolicity of metric spaces and graphs with $\delta$-slim geodesic
triangles.

\begin{proposition}\label{prop-soto}\cite{Soto-PhD}
If all geodesic triangles of a geodesic metric space (X,d) are
$\delta$-slim, then $X$ is $2\delta$-hyperbolic.
If all geodesic triangles of a graph $G$ are $\delta$-slim, then $G$
is $(2\delta + \frac{1}{2})$-hyperbolic.
\end{proposition}

An interval $I(u,v)$ of a graph $G$ is called $\nu$-{\it thin}, if $d(x,y)\le \nu$ for any two vertices $x,y\in I(u,v)$ such that $d(u,x)=d(u,y)$ and $d(v,x)=d(v,y).$
From the definition of $\delta$-hyperbolicity easily follows that intervals of a $\delta$-hyperbolic graph are $2\delta$-thin.

We note that a converse of this result holds too. If $G$ is a graph, denote by $G'$ the graph obtained by subdividing all edges of $G$. Papasoglu \cite{Pa} showed  that if $G'$ has $\nu$-thin intervals then $G$ is $f(\nu)$-hyperbolic for some function $f$. It is not clear what is the best possible $f$ for which this holds. Chatterji and Niblo in \cite{Cha-Ni} showed that $f$ can be taken to be a double exponential function. It would be interesting to have examples showing what is the dependence between $\delta,\nu$, e.g. whether it should be possible to show
that $f$ grows faster than linearly.

However, the following result holds:

\begin{proposition} \cite{ChDrEsHaVa} \label{mu-nu} If $G$ is a graph in which all intervals are $\nu$-thin and the metric triangles of $G$ have
sides of length at most $\mu,$ then  $G$ is $(16\nu+4\mu)$-hyperbolic.
\end{proposition}

Now, we recall the definition of hyperbolicity via the linear
isoperimetric inequality. Although this (combinatorial) definition of
hyperbolicity is given for geodesic metric spaces, it is quite common
to approximate the metric space by a graph via a quasi-isometric
embedding and to define $N$-fillings for the resulting graph (see for
example, \cite[pp. 414--417]{BrHa}). Since in this paper we deal only with graphs,
we directly give the definitions in the setting of graphs.

In a graph $G = (V,E)$, a \emph{loop} $c$ is a sequence of vertices
$(v_0,v_1,v_2,\ldots,v_{n-2},v_{n-1},v_0)$ such that for each $0 \leq
i \leq n-1$, either $v_i = v_{i+1}$, or $v_iv_{i+1} \in E$; $n$ is
called the \emph{length} $\ell(c)$ of $c$. A \emph{simple cycle} $c =
(v_0,v_1,v_2,\ldots,v_{n-2},v_{n-1},v_0)$ is a loop such that for all
$0 \leq i < j \leq n-1$, $v_i \neq v_j$.

A \emph{non-expansive map} $\Phi$ from a graph $G=(V,E)$ to a graph
$G'=(V',E')$ is a function $\Phi\colon V\to V'$ such that for all $v, w
\in V$, if $vw \in E$ then either $\Phi(v)=\Phi(w)$ or $\Phi(v)\Phi(w)
\in E'$. Note that a map $\Phi$ from $G$ to $G'$ is non-expansive if
and only if for all vertices $v,w$ of $G$, $d_{G'}(\Phi(v),\Phi(w))
\leq d_G(v,w)$.

For an integer $N>0$ and a loop $c = (v_0,v_1,v_2,\ldots, v_{n-2},
v_{n-1}, v_0)$ in a graph $G$, an \emph{$N$-filling} $(D,\Phi)$ of $c$
consists of a 2-connected planar graph $D$ and a non-expansive map
$\Phi$ from $D$ to $G$ such that the following conditions hold (see Figure~\ref{fig-N-filling} for an example):
\begin{enumerate}
\item the external face of $D$ is a simple cycle $(v'_0, v'_1,
  \ldots,v'_{n-1},v'_0)$ such that $\Phi(v_i') = v_i$ for all $0 \leq
  i \leq n-1$,
\item every internal face of $D$ has at most $2N$ edges.
\end{enumerate}

\begin{figure}[h]
\centering\includegraphics[scale=0.7]{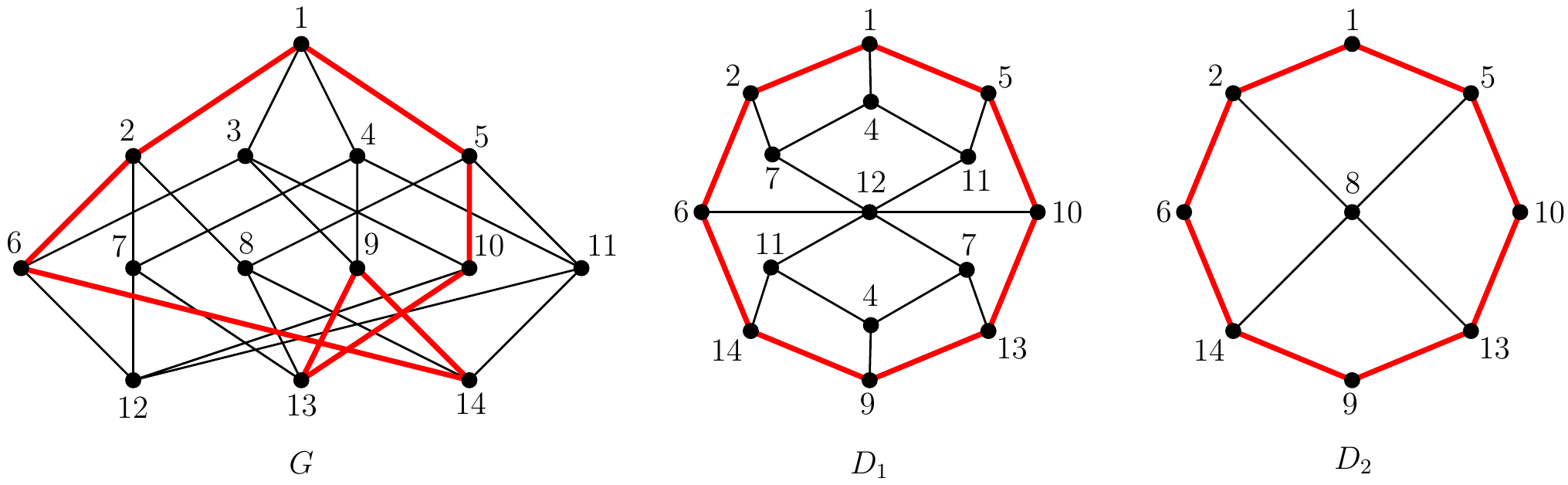}
\caption{Two different $2$-fillings $D_1$ and $D_2$ of the loop $c =
  (1, 2, 6, 14, 9, 13, 10, 5, 1)$ of $G$. $D_2$ is a $2$-filling of
  $c$ with a minimum number of faces and thus Area$_2(c) = 4$.}
\label{fig-N-filling}
\end{figure}

The $N$-{\it area} Area$_N(c)$ of $c$ is the minimum number of faces
in an $N$-filling of $c$.  A graph $G$ satisfies a {\it linear
  isoperimetric inequality} if there exists an $N>0$ such that any
loop $c$ of $G$ has an $N$-filling and Area$_N(c)$ is linear in the
length of $c$ (i.e., there exists a positive integer $K$ such that
Area$_N(c)\le K\cdot \ell(c)$).  The following result of
Gromov~\cite{Gr} proven in~\cite{AlBrCoFeLuMiShSh,Bowd,BrHa,Olsh}  is
the basic ingredient of our proof:

\begin{theorem} [Gromov] \label{isoperimetric} If a graph $G$ is $\delta$-hyperbolic, then any edge-loop of $G$ admits a $16\delta$-filling of
linear area. Conversely, if a graph $G$ satisfies the linear
isoperimetric inequality Area$_N(c)\le K\cdot \ell(c)$ for some
integers $N$ and $K$, then $G$ is $\delta$-hyperbolic, where
$\delta\le 108K^2N^3+9KN^2$.
\end{theorem}





\subsection{Graphs of ${\mathcal C}{\mathcal W}{\mathcal F}{\mathcal R}(s,s')$ and $(s,s')^*$-dismantlability}

A (non-necessarily finite) graph $G=(V,E)$ is called $(s,s')$-{\it dismantlable}
if the vertex set of $G$ admits a well-order $\preceq$ such that
for each vertex $v$ of $G$ there exists another vertex $u$ with $u \preceq v$
such that $B_{s}(v,G -\{ u\})\cap X_v\subseteq
B_{s'}(u,G),$ where $X_v:=\{ w\in V: w\preceq v\}.$
In the following, if $B_{s}(v,G-\{ u\})\cap X_v\subseteq B_{s'}(u,G),$ then we will say that $v$ {\it
is eliminated} by $u$ or that $u$ {\it eliminates} $v$. From the definition immediately follows that
if  $G$ is  $(s,s')$-dismantlable, then $G$ is also $(s,s'')$-dismantlable for any $s''>s'$ (with the same dismantling order).
In case of finite graphs, the following result holds (if $s=s'=1,$ this is the classical characterization of
cop-win graphs by Nowakowski, Winkler \cite{NowWin} and
Quilliot \cite{Qui83}):

\begin{theorem} \cite{ChChNiVa} \label{copwin} For any $s,s'\in {\mathbb N}\cup \{ \infty\},$ $s'\leq s$,
a finite graph $G$ belongs to the class  ${\mathcal C}{\mathcal W}{\mathcal F}{\mathcal R}(s,s')$
if and only if $G$ is $(s,s')$-dismantlable.
\end{theorem}

We will also consider a stronger version  of $(s,s')$-dismantlability: a graph $G$ is
$(s,s')^*$-{\it dismantlable} if the vertex set of $G$ admits a well-order $\preceq$ such that
for each vertex $v$ of $G$ there exists another vertex $u$ with $u\preceq v$ such that $B_{s}(v,G)\cap X_v\subseteq B_{s'}(u,G).$


In~\cite{ChChNiVa}, using a result from~\cite{ChEs}, it was shown that
$\delta$-hyperbolic graphs are $(s,s')^*$-dismantlable for some values
$s, s'$ depending of $\delta$. For sake of completeness, we recall
here these results.  The following proposition is a particular case of
Lemma~1 from~\cite{ChEs}.

\begin{proposition}\cite{ChEs}\label{prop-VB}
Let $G$ be a $\delta$-hyperbolic graph and $r$ be a non-negative integer.
Let $x,y,z$ be any three vertices of $G$ such that $d(y,z) \leq d(x,z)$
and $d(x,y) \leq 2r$. Then for any vertex $c \in I(x,z)$ such that
$d(x,c)=\min \{r,d(x,z)\}$, the inequality $d(c,y) \leq r+2\delta$ holds.
\end{proposition}

\begin{proof}
If $d(x,z) \leq r$, then $c=z$, and $d(c,y) \leq d(x,z) \leq r \leq
r+2\delta$. Suppose now that $d(x,z) > r$. Since $G$ is
$\delta$-hyperbolic, $d(c,y)+d(x,z) \leq \max
\{d(c,z)+d(x,y),d(c,x)+d(y,z)\} +2\delta$. Note that $d(c,z)+d(x,y)
\leq d(c,z)+2r = d(x,z)+r$ and $d(c,x) + d(y,z) \leq r + d(x,z)$. Consequently, $d(c,y) \leq r+ 2\delta$.
\end{proof}

The following corollary is an immediate consequence of
Proposition~\ref{prop-VB}.

\begin{corollary}\cite{ChChNiVa}\label{cor-delta-implies-copwin}
For a $\delta$-hyperbolic graph $G$ and any integer $r \ge \delta$,
any breadth-first search order $\preceq$ is a
$(2r,r+2\delta)^*$-dismantling order of $G$.
\end{corollary}

\section{Main result}

In this section we will prove that $(s,s')^*$-dismantlable graphs with
$s'< s$ are hyperbolic. All graphs occurring in the following result
are connected but not necessarily finite.

\begin{theorem} \label{theorem-main} If a graph $G$ is
  $(s,s')^*$-dismantlable with $0<s'<s$, then $G$ is
  $\delta$-hyperbolic with $\delta = 16(s+s')\left\lceil
  \frac{s+s'}{s-s'}\right\rceil + \frac{1}{2}\leq
  32\frac{s(s+s')}{s-s'} + \frac{1}{2}$.
\end{theorem}

\begin{proof} At the first step,  we will establish that for any
   cycle $c$ of $G$, Area$_{s+s'}(c)\le
   \left\lceil\frac{\ell(c)}{2(s-s')}\right\rceil$. In this part, we
   will follow the proofs of Lemma 2.6 and Proposition 2.7 of
   \cite[Chapter III.H]{BrHa} for hyperbolic graphs. At the second step, we will
   present a modified and refined proof of Theorem 2.9 of \cite[Chapter III.H]{BrHa},
   which will allow us to deduce that if Area$_{s+s'}(c)\le
   \frac{\ell(c)}{2(s-s')}+1$, then $G$ is
   $O(\frac{s^2}{s-s'})$-hyperbolic.

\begin{lemma} \label{shortenning} If a graph $G$ is
  $(s,s')^*$-dismantlable  with $s'<s$ and
  $c=(v_0,v_1,\ldots,v_{n-1},v_0)$ is a loop of $G$ of length $n>2(s+s'),$
  then $c$ contains two vertices $x=v_{p},y=v_q$ with $q-p=2s  \mod n$
  such that $d(x,y)\le 2s'$. 
\end{lemma}

\begin{proof}   Let $\preceq$ be an $(s,s')^*$-dismantling well-order of
  the vertex-set $V$ of $G$ and let $v$ be the largest element of the
  vertex-set of $c$ in this order.  Let $u$ be a vertex of $G$ that
  eliminates $v$ in $\preceq$. Without loss of generality suppose that
  $v=v_i,$ where $i>s$ and $i<n-s$. Let $x=v_{i-s}$ and
  $y=v_{i+s}$. Since $d(v,x)\le s,d(v,y)\le s,$ and $x,y\in X_v,$ from
  the definition of $(s,s')^*$-dismantlability we conclude that
  $d(u,x)\le s'$ and $d(u,y)\le s'.$ By triangle inequality,
  $d(x,y)\le 2s'$. Since $x=v_{i-s}$ and $y=v_{i+s}$, we obtain that
  $(i+s)-(i-s) = 2s$, and we can set $p:=i-s$ and $q:=i+s$.
\end{proof}

\begin{proposition} \label{linear_area} If a graph $G$ is
  $(s,s')^*$-dismantlable with $s'<s$ and $c$ is a
  loop of $G$, then {\rm Area}$_{s+s'}(c)\le
  \left\lceil\frac{\ell(c)}{2(s-s')}\right\rceil.$
\end{proposition}

\begin{proof} Let $c =(v_0, v_1, \ldots, v_{n-1}, v_0)$ be a
   loop of $G$. To prove that Area$_{s+s'}(c)\le
   \left\lceil\frac{\ell(c)}{2(s-s')}\right\rceil$, it suffices to
   show that there exists a 2-connected planar graph $D$ and a
   non-expansive map  $\Phi$ from $D$ to $G$ such that

\begin{itemize}
\item[(F1)] $D$ has at most
  $\left\lceil\frac{\ell(c)}{2(s-s')}\right\rceil$ faces,
\item[(F2)] all internal faces of $D$ have length at most $2(s+s')$,
\item[(F3)] the external face of $D$ is a simple cycle $(v'_0, v'_1,
  \ldots,v'_{n-1},v'_0)$ such that  $\Phi(v_i') =
  v_i$ for all $0  \leq i \leq n-1$.
\end{itemize}
 The image of each face of $D$ will be a loop of $G$ of length at
 most $2(s+s')$.

We proceed by induction on the length $n:=\ell(c)$ of $c$. 
If $n\le 2(s+s'),$ let $D$ consists
of a single face bounded by a simple cycle
$(v_0',v_1',\ldots,v_{n-1}',v_0')$ of length $n$ and for each $i$, let
$\Phi(v_i') = v_i$. This shows that Area$_{s+s'}(c)=1$.

Now, suppose that $n>2(s+s')$. By Lemma \ref{shortenning} there exist
two vertices $x=v_p,y=v_q$ of $c$ with $q-p =2s \mod n$ and $d(x,y)\le
2s'.$ Suppose without loss of generality that $q = p+2s$. Let
$P'=(x=v_{p},v_{p+1},\ldots, v_{q-1},v_q=y)$ and
$P''=(x=v_{p},v_{p-1},\ldots,v_0,v_{n-1},\ldots,v_{q+1},v_q=y)$ be the
two $(x,y)$-paths constituting $c$.  If $x = y$, let $P = (x,y)$; if
$x \neq y$, let $P=(x = w_0, w_1, \ldots, w_k=y)$ be any shortest path
in $G$ between $x$ and $y$. Note that $\ell(P) \leq 2s' < 2s =
\ell(P')$.



Let $c_0$ be the loop obtained as the concatenation of the paths $P$
from $x$ to $y$ and $P'$ from $y$ to $x$. Since $\ell(P') = 2s$ and
$\ell(P) \leq 2s'$, we have $\ell(c_0)\le 2s+2s'$.  Let $c_1$ be the
loop obtained as the concatenation of the paths $P''$ from $y$ to
$x$ and $P$ from $x$ to $y$.  Note that
$\ell(c_1) = \ell(P)+\ell(P'') \le \ell(P)+\ell(c)-\ell(P')\le
\ell(c)-(2s-2s') < \ell(c)$.



By induction assumption, $c_1$ admits an $(s+s')$-filling
$(D_1,\Phi_1)$ satisfying the conditions (F1),(F2), and (F3). Note
that the external face of $D_1$ is bounded by a cycle $(v_p' = x' =
w_0', w_1', \ldots, w_{k}'= y' =
v_q',v_{q+1}',\ldots,v_{n-1},v_0,\ldots,v_{p-1}',v_p')$ such that
 $\Phi_1(v_i') = v_i$ for all $i \in [0,p] \cup [q,n-1]$  and
$\Phi_1(w_i') = w_i$ for all $0\le i\le k$.

Consider the planar graph $D$ obtained from $D_1$ by adding $q-p-1$
new vertices forming a path $(x'=v_p', v_{p+1}',\ldots,v_q'=y')$ from
$x'$ to $y'$ on the external face of $D_1$ such that the external face
of $D$ is bounded by the cycle $(v_0',v_1',\ldots,v_{n-1}',v_0')$.
Let $\Phi$ be the non-expansive map defined by $\Phi(v') =
\Phi_1(v')$ for every $v \in V(D_1)$ and $\Phi(v_i') = v_i$ for every
$p+1 \leq i \leq q-1$. Clearly, $D_1$ is a 2-connected planar graph and
for each $0 \leq i \leq n-1$ we have $\Phi(v_i') = v_i$.  The planar graph
$D$ has one more internal face than $D_1$ that is bounded by the cycle
$(x'=v_p',v_{p+1}',\ldots,v_q'=y'=w_{k}',w_{k-1}',\ldots,w_{1}',w_{0}'=x')$. This
cycle has the same length as $c_0$ and is thus bounded by
$2(s+s')$. Consequently, $(D,\Phi)$ satisfies the conditions (F2) and
(F3).

It remains to show that the $(s+s')$-filling $(D,\Phi)$ of $c$
satisfies (F1). Since $\ell(c_1)\le \ell(c)-2(s-s')$, by induction
assumption, we obtain
$${\rm Area}_{s+s'}(c)\le {\rm Area}_{s+s'}(c_1)+1\le \left\lceil \frac{\ell(c_1)}{2(s-s')}\right\rceil+1\le \left\lceil\frac{\ell(c)-2(s-s')}{2(s-s')}\right\rceil+1= \left\lceil\frac{\ell(c)}{2(s-s')}\right\rceil,$$
yielding the desired inequality.
\end{proof}

Now, we revisit the proof of Theorem 2.9 of~\cite[Chapter III.H]{BrHa} that
corresponds to Theorem~\ref{isoperimetric} stated above.  Namely, we
extend this result to the case of rational $K$  and
improve its statement by showing that the hyperbolicity
of $G$ is quadratic (and not cubic) in $N$.

We start with an auxiliary result.  For a subset of vertices
$A\subseteq V$ of a graph $G=(V,E)$ and an integer $k\ge 0,$ let
$B_k(A,G)=\{ v\in V: d_G(v,A)\le k\}$ denote the $k$-{\it
  neighborhood} of $A$ in $G$.

\begin{lemma}\label{lem-lower-bound-faces}
Let $G$ be a graph and $k>0$ be an integer.  Consider a simple
cycle $c = (v_0, v_1, \ldots, v_{n-1})$ of $G$ and two integers $p, q$
such that $k < p < p +2k \leq q < n-k$, $d_G(v_p,v_q) = q-p$, and
$B_k(\{v_p,v_{p+1},\ldots,v_q\},G)\cap c
=\{v_{p-k},v_{p-k+1},\ldots,v_{q+k-1},v_{q+k}\}$.  Consider any
$N$-filling $(D,\Phi)$ of $c$ and let $c' = (v'_0, v'_1,
\ldots,v'_{n-1},v'_0)$ be the cycle bounding the external face of $D$
(where $\Phi(v_i') = v_i$ for all $0\le i\le n-1$). Then, in the
subgraph of $D$ induced by $B_k(\{v_p',v_{p+1}',\ldots,v_q'\},D)$, there
exist at least $\frac{k(q-p-2k)}{N^2}$ faces of $D$ that contain at
most one vertex at distance $k$ from $\{v_p',v_{p+1}',\ldots,v_q'\}$.
\end{lemma}




\begin{proof}
Let $(D,\Phi)$ be an $N$-filling of $c$. Since $\Phi$ is a
non-expansive map from $D$ to $G$, for any vertices $u',v' \in V(D)$,
the distance $d_D(u',v')$ in $D$ is greater than or equal to the
distance $d_G(\Phi(u'),\Phi(v'))$ in $G$ between their images. Note
also that $\ell(c) \geq q-p+2k+1$.

Let $F(D)$ be the set of faces of $D$.  We define recursively a set
$V_i \subseteq V(D)$ of vertices, a set $E_i \subseteq E(D)$ of edges,
and a set $F_i \subseteq F(D)$ of faces of $D$. Let $P_0 = (V_0,E_0)$
be the path $(v'_p, v'_{p+1}, \ldots, v'_q)$.  Let $F_0$ be the set of
faces of $D$ that contain vertices of $V_0$.

For any $i \geq 1$, let $V_{i}$ (resp., $E_{i}$) be the set of
vertices (resp., edges) belonging simultaneously to faces of
$F_{i-1}$ and to faces of $F(D) \setminus (F_0 \cup \ldots \cup
F_{i-1})$. Let $F_i$ be the set of faces of $F(D) \setminus (F_0 \cup
\ldots \cup F_{i-1})$ containing vertices of $V_i$.  Since $D$ is a
planar graph, for each $i$, the connected components of the graph
$H_i:=(V_i,E_i)$ are (non-necessarily simple) paths and cycles. The
vertices of $c'\cap V_i$ necessarily belong to a single path of $H_i$
(again, this follows from the planarity of $D$), which we will denote
by $P_i$.

Since each face of $D$ has length at most $2N$, all vertices appearing in
a face of $F_i$ are at distance at most $N$ from $V_i$. Moreover, each
face of $F_i$ contains at most one vertex at distance $N$ from
$V_i$. Consequently, each face of $F_i$ contains only vertices at distance
at most $(i+1)N$ from $V_0 = \{v'_p, v'_{p+1}, \ldots, v'_q\}$, and
each face of $F_i$ contains at most one vertex at distance $(i+1)N$
from $V_0$.

Assume now that $i \leq k/N$, and let $F = F_0 \cup F_1 \cup \ldots
\cup F_{i-1}$. Consider a vertex $v_j'$ of the external face of $D$
that belongs to a face of $F$. Since $d_G(v_j,\Phi(V_0))\le
d_D(v_j',V_0) \leq iN \leq k$, and since $B_k(\Phi(V_0),G) \cap
c=\{v_{p-k},v_{p-k+1},\ldots,v_{q+k-1},v_{q+k}\}$, we have $p-k \leq j
\leq q+k$.  Consequently, $v'_0$ does not appear on a face of $F$.
Therefore, $V_{i}$ and $E_{i}$ are non-empty and $P_{i}$ is a
well-defined path of $H_{i}$.  Let $\ell \leq p$ be the largest index
such that $v_{\ell-1}'$ does not belong to a face of $F$; we know that
$p-k \leq \ell \leq p$.  Similarly, let $j \geq q$ be the smallest
index such that $v_{j+1}'$ does not belong to a face of $F$; we know
that $q \leq j \leq q+k$. Since $\Phi$ is non-expansive,
$d_D(v_p',v_q') \geq d_G(v_p,v_q) = q-p$ and since $(v_p',v_{p+1}',
\ldots, v_q')$ is a path of length $q-p$ from $v_p'$ to $v_q'$ in $D$,
$d_D(v_p',v_q')=q-p$. By triangle inequality, $d_D(v_\ell',v_j') \geq
d_D(v_p',v_q') - d_D(v_p',v_\ell') - d_D(v_q',v_j') \geq q-p-2k$.  Let
$H'_i$ be the subgraph of $D$ induced by the vertices appearing on
faces of $F_i$. Since $P_i$ is a subgraph of $H'_i$, there is a path
from $v_\ell'$ to $v_j'$ in $H_i'$. Let $P'_i = (v_\ell'=w_0', \ldots,
w_t' = v_j')$ be a shortest $(v_\ell',v_j')$-path in $H'_i$.  Note
that if two vertices $w_{j_1}', w_{j_2}'$ of $P_i'$ belong to a common
face of $F_i$, then $d_{H_i'}(w_{j_1}',w_{j_2}') \leq N$ and
consequently, $|j_2 - j_1| \leq N$. Consequently, the vertices of
$P_i'$ belong to at least $t/N$ distinct faces of $F_i$.  Since $t =
d_{H_i'}(v_\ell',v_j') \geq d_D(v_\ell',v_j') \geq q - p- 2k$, this
implies that there are at least $(q-p-2k)/N$ faces in $F_i$.

Therefore, if we set $i:= k/N$ and consider the number $A$ of faces in
$F=F_0 \cup F_1 \cup \ldots \cup F_{i-1}$, we get that $A \geq k(q- p
- 2k)/N^2$.  Note that each face of $F$ contains only vertices at
distance at most $iN=k$ from $V_0$, and every face of $F$ contains at
most one vertex at distance $k$ from $V_0$.
\end{proof}

\begin{proposition}\label{prop-isoperimetric-implies-hyperbolic}
For a graph $G$ and constants $K \in \mathbb{Q}$ and $N \in
\mathbb{N}$ such that $2KN$ is a positive integer, if for every cycle
$c$ of $G$, Area$_N(c) \leq \lceil K l(c)\rceil $, then the geodesic triangles
of $G$ are $16KN^2$-slim and $G$ is $(32KN^2+\frac{1}{2})$-hyperbolic.
\end{proposition}

\begin{proof}
Using the fact that if all geodesic triangles of $G$ are $\delta$-slim
then $G$ is $2\delta+\frac{1}{2}$-hyperbolic
(Proposition~\ref{prop-soto}), we will show that under our conditions,
all geodesic triangles of $G$ are $16 KN^2$-slim.  This proof mainly
uses ideas and notations from the proof of Theorem~2.9
of~\cite[Chapter III.H]{BrHa}.

Let $k:= 2KN^2 $ and assume that there exist three (graph) geodesic segments
$[p,q]$, $[q,r]$ and $[p,r]$ forming a geodesic triangle
$\Delta(p,q,r)$ of $G$ that is not $8k$-slim, i.e., there exists a
vertex $v \in [p,q]$ such that $d(v,[p,r]\cup[q,r]) > 8k$. Exchanging
the roles of $p$ and $q$ if necessary, there are two cases to consider
(see Figure~\ref{fig-iso-implies-hyperbolic}):
\begin{itemize}
\item either $d([p,v],[q,r]) > 2k$ and $d([v,q],[p,r] > 2k$,
\item or there exists $w \in [v,q]$
such that $d(w,[p,r]) \leq 2k$.
\end{itemize}

\begin{figure}[t]
\centering\includegraphics[scale=0.8]{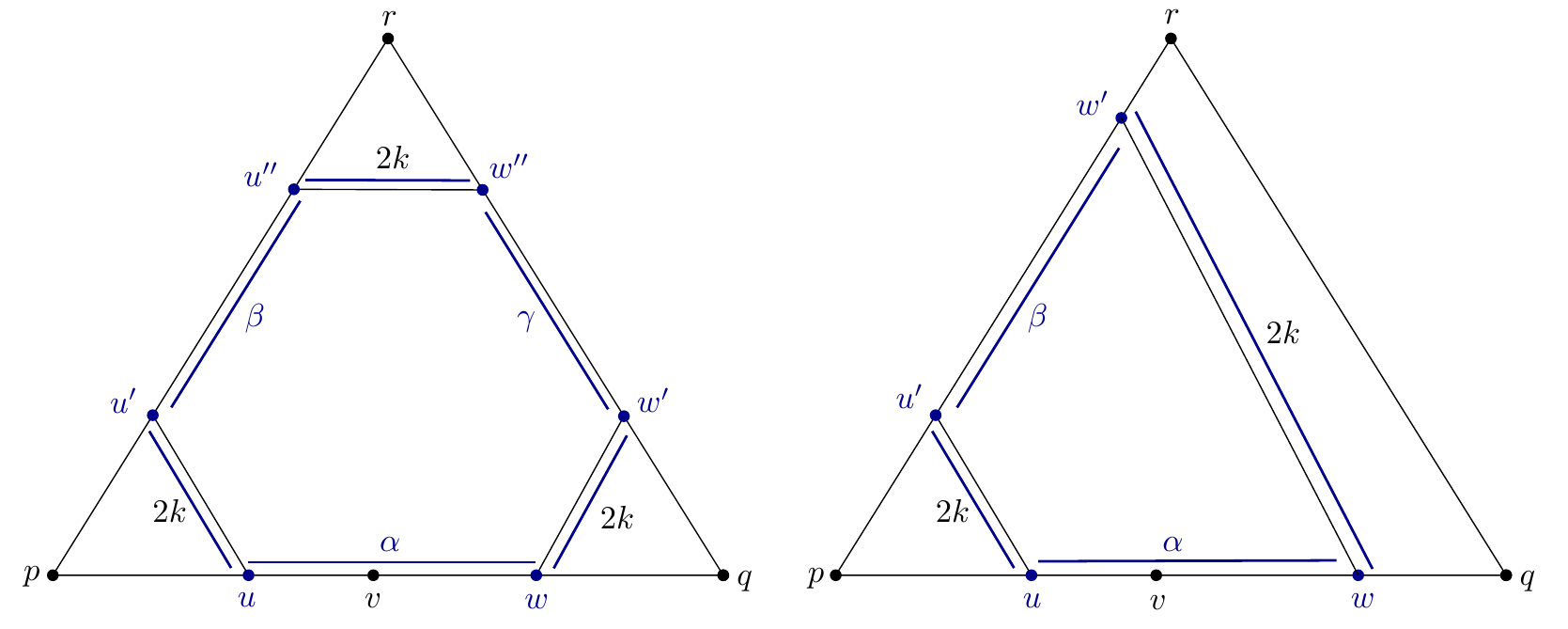}
\caption{Cases 1 and 2 of the proof of Proposition \ref{prop-isoperimetric-implies-hyperbolic}}
\label{fig-iso-implies-hyperbolic}
\end{figure}

\noindent\textbf{Case 1.} (see Figure~\ref{fig-iso-implies-hyperbolic}, left)
$d([p,v],[q,r]) > 2k$ and $d([v,q],[p,r] > 2k$.
\smallskip

In this case, let $u \in [p,v]$ be the closest vertex to $v$ such that
$d(u,[p,r])=2k$ and let $u'\in [p,r]$ be the closest vertex to $r$
such that $d(u,u')=2k$. Let $w \in [v,q]$ be the closest vertex to $v$
such that $d(w,[q,r])=2k$, and let $w'\in [q,r]$ be the
closest vertex to $r$ such that $d(w,w')=2k$.
We denote by $[u',r]$ (resp. $[w',r]$) the subgeodesic of $[p,r]$
(resp. $[q,r]$) from $u'$ (resp. $w'$) to $r$.  Let $u'' \in [u',r]$
be the closest vertex from $u'$ such that $d(u'',[w',r]) \leq 2k$ and
let $w''\in [w',r]$ be the closest vertex from $w'$ such that
$d(u'',w'') \leq 2k$.  We denote by $[u',u'']$ (resp. $[w',w'']$) the
subgeodesic of $[u',r]$ (resp. $[w',r]$) from $u'$ (resp. $w'$) to
$u''$ (resp. $w''$).  Let $[u,u']$ (resp. $[w,w']$, $[u'',w'']$) be a
geodesic from $u$ to $u'$ (resp. from $w$ to $w'$ and from $u''$ to
$w''$).  Let $\alpha = d(u,w)$, $\beta= d(u',u'')$, and $\gamma =
d(w',w'')$.  Since $d(v,[p,r]\cup[q,r]) > 8k$, $\alpha >12k$; since
$[u,w]$ is a shortest path, $\beta+\gamma > 6k$.  Due to our choice of
$u,u',u'',w,w',w''$, $d([u,w],[u',u''])= d([u,w],[w',w'']) = 2k$. Note
that if $u' \neq u''$, then $d([u',u''],[w',w'']) = 2k$ and that if $u' =
u''$, then $d(u',w') \geq d(u,w) - d(u,u') - d(w,w') = \alpha - 4k >
8k$. In both cases, $d([u',u''],[w',w'']) = d(u'',w'') = 2k$.

Let $c$ be the simple cycle of $G$ obtained as the concatenation of
the six geodesics $[u,u'], [u',u''], [u'',w''], [w'',w'], [w',w],$ and
$[w,u]$ (see Fig. \ref{fig-iso-implies-hyperbolic}, left). From the
definition of the vertices $u,u',u'',w'',w',w$ it follows that $c$ is
a simple cycle.

Moreover, if there exists $x \in [u,w]$ such that $d(x,[u'',w'']) \leq
k$, then either $d(x,u'') \leq 2k$, or $d(x,w'') \leq 2k$. In the
first case, from the definition of $u'$ and $u$, it implies that
$u'=u''$ and that $x = u$. Analogously, in the second case, it implies
that $w'=w''$ and $x=w$. Consequently, the only vertices of $c$
appearing in $B_k([u,w],G)$ are the vertices at distance at most $k$
from $[u,w]$ on $c$.  Similarly, if there exists $x \in [w',w'']$ such
that $d(x,[u,u']) \leq k$, then either $d(x,u) \leq 2k$, or $d(x,u')
\leq 2k$. In the first case, it contradicts $d([p,v],[q,r]) > 2k$. In
the second case, from our choice of $u''$ and $w''$, it implies that
$u'=u''$ and $x=w''$. Consequently, the only vertices of $c$ appearing
in $B_k([w',w''],G)$ are the vertices at distance at most $k$ from
$[w',w'']$ on $c$. For the same reasons, the only vertices of $c$
appearing in $B_k([u',u''],G)$ are the vertices at distance at most
$k$ from $[u',u'']$ on $c$. This proves that $c$ satisfies the
conditions of Lemma \ref{lem-lower-bound-faces} for $[u,v]$,
for $[u',u'']$ if $\beta > 2k$, and for $[w',w'']$ if $\gamma > 2k$.

Let $(D,\Phi)$ be an $N$-filling of $c$. We want to get a lower bound
on the number of faces of $D$ (and therefore, a lower bound on
Area$_N(c)$). In the planar graph $D$, we denote by $\Phi^{-1}([u,w])$
(resp., $\Phi^{-1}([u',u''])$ and $\Phi^{-1}([w',w''])$) the
unique path on the boundary of $D$ that is mapped to $[u,w]$
(resp., $[u',u'']$ and $[w',w'']$).

Let $F_\alpha$ denote the set of faces of $D$ that contain only
vertices in $B_k(\Phi^{-1}([u,w]),D)$ and at most one vertex at
distance $k$ from $\Phi^{-1}([u,w])$.  Similarly, let $F_\beta$ (resp.
$F_\gamma$) be the set of faces of $D$ that contain only vertices in
$B_k(\Phi^{-1}([u',u'']),D)$ (resp. in $B_k(\Phi^{-1}([w',w'']),D)$)
and at most one vertex at distance $k$ from $\Phi^{-1}([u',u''])$
(resp.  from $\Phi^{-1}([w',w''])$). Note that from our choice of $u$
and $u'$, if $s \in [u,w]$, $t \in [u',u'']$, and $d(s,t) = 2k$, then
$s = u$ and $t = u'$. Consequently, there is no face in $F_\alpha \cap
F_\beta$. Similarly, $F_\alpha \cap F_\gamma=F_\beta \cap
F_\gamma=\emptyset$. Moreover, there is no edge appearing in a face
$f$ of $F_\alpha$ and $f'$ of $F_\beta$ since both endvertices of this
edge should be at distance $k$ from $[u,w]$ and $[u',u'']$. Similarly,
no edge appears in a face of $F_\alpha$ (resp. $F_\beta$) and in a
face of $F_\gamma$.  Consequently, the number of faces of $D$ is at
least $|F_\alpha|+|F_\beta|+|F_\gamma|+1$.

From Lemma~\ref{lem-lower-bound-faces}, since $\alpha > 2k$,
$F_\alpha$ contains at least $ \frac{k(\alpha-2k)}{N^2} = 2K
(\alpha-2k)$ faces. Similarly, if $\beta > 2k$ (resp. $\gamma > 2k$),
then $|F_\beta| \geq 2K(\beta-2k)$ (resp. $|F_\gamma| \geq 2K(\gamma -
2k)$). Note that if $\beta \leq 2k$ (resp. $\gamma \leq 2k$), then
$|F_\beta| \geq 0 \geq 2K(\beta-2k)$ (resp. $|F_\gamma| \geq 0 \geq
2K(\gamma - 2k)$).  Consequently, Area$_{N}(c) \geq
2K(\alpha+\beta+\gamma-6k) +1$. By the isoperimetric inequality,
Area$_N(c) \leq K \ell(c) + 1 = K (\alpha + \beta + \gamma + 6k) +
1$. From these two inequalities, we get that $\alpha+\beta+\gamma \leq
18 k$, contradicting the fact that $\alpha > 12k$ and $\beta+\gamma >
6k$.

\medskip\noindent\textbf{Case 2.} (see Figure~\ref{fig-iso-implies-hyperbolic}, right)
There exists $w \in [v,q]$ such that $d(w,[p,r]) \leq 2k$.
\smallskip

In this case, let $w \in [v,q]$ be the closest vertex to $v$ such that
$d(w,[p,r]) = 2k$; let $w' \in [p,r]$ be the closest vertex to $p$
such that $d(w,w') = 2k$. We denote by $[p,w']$ the subgeodesic of
$[p,r]$ from $p$ to $w'$.  Let $u \in [p,v]$ be the closest vertex to
$v$ such that $d(u,[p,w'])=2k$ and let $u'\in [p,w']$ be the closest
vertex to $w'$ such that $d(u,u')=2k$. Let $[u,u']$ (resp. $[w,w']$)
be a geodesic from $u$ to $u'$ (resp. from $w$ to $w'$).  Let $\alpha
= d(u,w)$ and $\beta= d(u',w')$.  Due to our choice of $w,w',u$ and
$u'$, $d([u,w],[u',w']) =2k$; since $d(v,[p,r]\cup[q,r]) > 8k$,
$\alpha >12k$; since $[u,w]$ is a shortest path, $\beta > 8k$.

Let $c$ be the simple cycle obtained as the concatenation of geodesics
$[u,u'], [u',w'], [w',w],$ and $[w,u]$.  As in Case~1, using
Lemma~\ref{lem-lower-bound-faces}, we show that Area$_{N}(c) \geq
2K(\alpha+\beta-4k) +1$. By the isoperimetric inequality, Area$_N(c)
\leq K \ell(c) + 1 = K (\alpha + \beta + 4k) + 1$. From these
inequalities, we get that $\alpha+\beta \leq 12 k$, contradicting the
fact that $\alpha > 12k$ and $\beta > 8k$.
\end{proof}

\begin{remark} \label{remark2}
The dependence of $\delta ,N$ in Proposition \ref{prop-isoperimetric-implies-hyperbolic} is the ``best possible'' in the following sense. There are graphs
$G_N$ ($N\in \mathbb N)$ which satisfy
 Area$_N(c) \leq \lceil l(c)\rceil $ and which are not $\delta $-hyperbolic for $\delta= o(N^2)$ (so $\delta $ in general grows quadratically in $N$).

Indeed, take $G_N$ to be a planar square $N\times N$ grid subdivided
into squares of side-length $N$ (see Figure~\ref{fig-grid-grid} for an
example with $N = 4$). Then clearly for every cycle $c$, Area$_{4N}(c)
\leq \dfrac{1}{4}\lceil l(c)\rceil $. Consider now the four corners
$a, b, c, d$ of the grid (see Figure~\ref{fig-grid-grid}); we have
$d(a,c)+d(b,d) = 4N^2 > 2N^2 = d(a,b)+d(c,d) = d(a,d)+d(b,c)$ and thus
$\delta \geq N^2$.
\end{remark}

\begin{figure}[t]
\centering\includegraphics[scale=1]{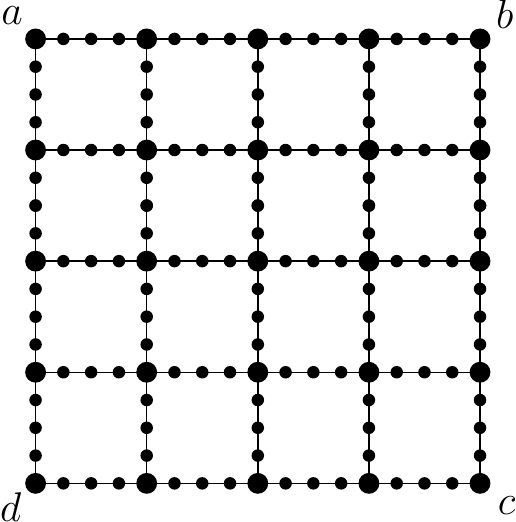}
\caption{The graph described in Remark~\ref{remark2} when $N = 4$.}
\label{fig-grid-grid}
\end{figure}

The assertion of Theorem \ref{theorem-main} follows from Propositions
\ref{linear_area} and \ref{prop-isoperimetric-implies-hyperbolic} by
setting $N:=s+s'$ and $K:=
\frac{1}{2N}\cdot\left\lceil\frac{N}{(s-s')}\right\rceil \geq \frac{1}{2(s-s')}$.
\end{proof}

Here are the main consequences of Theorem \ref{theorem-main}:

\begin{corollary} \label{cor1} If a graph $G$ is
  $(s,s')$-dismantlable with $s'<s$ (in particular, $G$ is a finite
  $(s,s')$-cop-win graph), then $G$ is $\delta$-hyperbolic with
  $\delta=64s^2.$
\end{corollary}

\begin{proof} Since  $(s,s')$-dismantlable graphs are also $(s,s-1)$-dismantlable, it is enough to prove our result for $s'=s-1$. Notice that a graph $G$ is $(s,s-1)$-dismantlable if and only if $G$ is  $(s,s-1)^*$-dismantlable. From Theorem \ref{theorem-main} with  $s-s'=1$, we conclude that $G$ is $64s^2$-hyperbolic.
\end{proof}

%

We do not have examples of finite $(s,s-1)$-dismantlable graphs whose
hyperbolicity is quadratic in $s$ and leave this as an \emph{open
  question}.  However, if $s-s'=\Omega(s),$ from Theorem
\ref{theorem-main} we immediately obtain that $\frac{s^2}{s-s'}=O(s),$
thus $\delta$ is linear in $s$.

\begin{corollary} \label{cor2}
If a graph $G$ is $(s,s')^*$-dismantlable with $s-s' \geq ks$
for some constant $k > 0$, then $G$ is
$\frac{64s}{k}$-hyperbolic. Conversely, if $G$ is
$\delta$-hyperbolic, then $G$ is $(2r,r+2\delta)^*$-dismantlable for
any $r>0$.
\end{corollary}

\begin{proof}
The first assertion follows directly from Theorem~\ref{theorem-main}.
The second assertion follows from Proposition ~\ref{prop-relations-delta}.
\end{proof}

%
%

Similarly to Theorem~\ref{isoperimetric} that characterizes
hyperbolicity via linear isoperimetric inequality,
Corollary~\ref{cor2} characterizes hyperbolicity via
$(s,s')^*$-dismantlability.

\section{Weakly modular graphs}

In this section, we consider weakly modular graphs; for this
particular class of graphs, we obtain stronger results than in the
general case: namely, we show that for any $s' < s$, if a weakly
modular graph $G$ is $(s,s')$-dismantlable, then $G$ is
$O(s)$-hyperbolic.

Many classes of graphs occurring in metric graph theory and geometric
group theory (in relationship with combinatorial nonpositive curvature
property) are weakly modular: median graphs (alias, 1-skeletons of
CAT(0) cube complexes), bridged and weakly bridged graphs (1-skeletons
of systolic and weakly systolic complexes), bucolic graphs
(1-skeletons of bucolic complexes), Helly graphs (alias, absolute
retracts), and modular graphs.  For definitions and properties of
these classes of graphs the interested reader can read the survey
\cite{BaCh_survey} and the paper \cite{BrChChGoOs}.

A graph $G$ is {\it weakly modular}~\cite{BaCh_weak,Ch_metric} if it
satisfies the following triangle and quadrangle conditions:
\begin{itemize}
\item
{\it Triangle condition}:  for any three vertices $u,v,w$ with
$1=d(v,w)<d(u,v)=d(u,w)$ there exists a common neighbor $x$ of $v$
and $w$ such that $d(u,x)=d(u,v)-1.$
\item
{\it Quadrangle condition}: for any four vertices $u,v,w,z$ with
$d(v,z)=d(w,z)=1$ and  $2=d(v,w)\le d(u,v)=d(u,w)=d(u,z)-1,$ there
exists a common neighbor $x$ of $v$ and $w$ such that
$d(u,x)=d(u,v)-1.$
\end{itemize}

All metric triangles of weakly modular graphs are
equilateral. Moreover, they satisfy a stronger equilaterality
condition:

\begin{lemma} \cite{Ch_metric} \label{lem-weakly-modular} A graph $G$ is weakly modular if and only if for any metric triangle $v_1v_2v_3$ of $G$ and any two vertices
$x,y\in I(v_2,v_3),$ the equality  $d(v_1,x)=d(v_1,y)$ holds.
\end{lemma}

\medskip

The following result shows that in the case of $(s,s')$-dismantlable weakly modular graphs the hyperbolicity is always a linear function of $s$ for all values of $s$ and $s'<s$.

\begin{theorem}\label{weakly-modular}  If $G$ is an $(s,s')$-dismantlable weakly modular graph with $s'<s$, then $G$ is $184s$-hyperbolic.
\end{theorem}

\begin{proof}
Since  $(s,s')$-dismantlable graphs are $(s,s-1)$-dismantlable, it is
enough to prove our result for $s'=s-1$.  We will establish our result
in two steps. First, we show that if all metric triangles of an
$(s,s-1)$-dismantlable graph $G$  are $\mu$-{\it bounded} (i.e., have
sides of length at most $\mu$), then all intervals of $G$ are $(4s+2\mu)$-thin. In the second step, we show that in  $(s,s-1)$-dismantlable weakly modular graphs  all metric triangles are $6s$-bounded.

We continue with some properties of general $(s,s-1)$-dismantlable
graphs. A subgraph $H=(V',E')$ of a graph $G$ is called a {\it locally
  $s$-isometric subgraph} of $G$ if $H$ contains a collection
$\mathcal P$ of geodesics of $G$ such that for any vertex $v$ of $H$
there exists a geodesic $P\in {\mathcal P}$ passing via $v$ and such
that the distances $d(v,x),d(v,y)$ in $G$ between $v$ and the
endvertices $x,y$ of $P$ are at least $s$. Note that this implies that
for every vertex $v \in V(H)$, there exists a subgeodesic $P_v$ of a
geodesic $P \in {\mathcal P}$ containing $v$ such that $d(v,x) =
d(v,y) = s$ in $G$.

\begin{lemma} \label{locally-isometric} 
If a graph $G$ is $(s,s-1)$-dismantlable, then $G$ does not contain
finite locally $s$-isometric subgraphs.
\end{lemma}

\begin{proof} Suppose by way of contradiction that $G$ contains a finite locally $s$-isometric subgraph $H=(V',E').$ Let $\preceq$ be an $(s,s-1)$-dismantling well-order of the vertex-set $V$ of $G$ and let $v$ be the largest  element of the set $V'$ in this order. Let $u$ be a vertex of $G$ that eliminates $v$ in $\preceq$. Since $H$ is locally $s$-isometric, $H$ contains a geodesic $P$ of $G$ of length $2s$ passing via $v$ such that $d(v,x)=d(v,y)=s$, where $x$ and $y$ are the endvertices of $P$. If $u\in P,$ say $u$ belongs to the subpath $P'$ of $P$ comprised between $v$ and $x,$ then the subpath $P''$ of $P$ between $v$ and $y$ is completely contained in $B_s(v,G-\{ u\})\cap X_v$. From the choice of $u$, we have $P''\subseteq B_{s-1}(u,G)$. Hence $d(u,y)\le s-1,$ which is impossible because $P$ is a geodesic of length $2s$ passing via $u$ and $u\in P'$. So, let $u\notin P$. Then $P$ is completely contained in $B_s(v,G-\{ u\})\cap X_v$, whence $P\subseteq B_{s-1}(u,G)$. In particular, $d(u,x)\le s-1$ and $d(u,y)\le s-1$. Since $d(x,y)=2s,$ we again obtain a contradiction.
\end{proof}

We will say that a cycle $c$ of $G$ is  \emph{$s$-geodesically
  covered} if there exists a set ${\mathcal P}=\{
P_0,P_1,\ldots,P_{n-1}\}$ of geodesics of $G$ such that (i) each $P_i$
is a subpath of $c$, (ii) each edge of $c$ is contained in a geodesic
of $\mathcal P$, (iii) if $P_i$ and $P_j$ are not consecutive (modulo
$n$), then $P_i$ and $P_j$ are edge-disjoint, and (iv) if $P_i$ and
$P_j$ are consecutive (i.e., $j=i+1 \mod n$), then $P_i\cap P_{j}$ is a path of length $\ge 2s$.

\begin{lemma} \label{no-cycles}  
If a graph $G$ contains a $s$-geodesically covered cycle, then it
contains a finite locally $s$-isometric subgraph. 

In particular, if a graph $G$ is $(s,s-1)$-dismantlable, then $G$ does
not contain $s$-geodesically covered cycles.
\end{lemma}

\begin{proof} Let $c$ be a 
  $s$-geodesically covered cycle of $G$ and let ${\mathcal P}=\{
  P_0,P_1,\ldots,P_{n-1}\}$ be the corresponding set of geodesics
  satisfying conditions (i)-(iv).  Fix a cyclic traversal of $c$.  For
  each $0\le i\le n-1$, let $x_i$ and $y_i$ be the end-vertices of
  $P_i$ labeled in such a way that following the traversal, $x_i$ and
  $y_i$ are the first and the last vertices of $P_i$.  

  Then $P_i\cap P_{i+1}$ is a geodesic between $x_{i+1}$ and $y_i$ (as
  a subgeodesic of $P_i$ and $P_{i+1}$).  By condition (iv), its
  length is at least $2s,$ whence $d(x_{i+1},y_i)\ge 2s.$ Analogously,
  $d(x_{i+2},y_{i+1})\ge 2s.$ On the other hand, since by (iii)
  $P_i\cap P_{i+2}$ does not contain edges, either $y_i=x_{i+2}$ or
  $x_{i+2}$ is located between $y_i$ and $y_{i+1}$.  Therefore, since
  $y_i$ and $x_{i+2}$ belong to the geodesic $P_{i+1}$, we obtain
  $d(y_i,y_{i+1})=d(y_i,x_{i+2})+d(x_{i+2},y_{i+1})\ge 2s.$

Pick any vertex $v$ of $c$. Since $c$ is covered by the geodesics of
$\mathcal P$, there exists at least one such geodesic that contains
$v$. Let $P_{i_0}$ be a geodesic of $\mathcal P$ selected in a such a
way that $v\in P_{i_0}$ and $k:=\min\{ d(v,x_{i_0}),d(v,y_{i_0})\}$ is
as large as possible. If $k <s$, assume without loss of generality
that $d(v,y_{i_0})<s$. By condition (iv) applied to the geodesics
$P_{i_0}$ and $P_{{i_0}+1},$ we conclude that
$d(x_{{i_0}+1},y_{i_0})\ge 2s.$ Since $d(v,y_{i_0})<s,$ necessarily
$v\in P_{{i_0}+1}$ and $d(v,x_{{i_0}+1})>s.$ On the other hand, since
$d(y_{i_0},y_{{i_0}+1})\ge 2s$ and $v$ is located on the geodesic
$P_{{i_0}+1}$ between $x_{{i_0}+1}$ and $y_{i_0}$, necessarily
$d(v,y_{{i_0}+1})\ge 2s,$ contrary to the choice of $P_{i_0}$ as the
path of $\mathcal P$ containing $v$ and maximizing $\min\{
d(v,x_{i_0}),d(v,y_{i_0})\}.$ Thus, $k\ge s$ for every choice of
$v$. Consequently, $c$ is a locally $s$-isometric subgraph of $G$,
establishing the first assertion of the lemma.  The second assertion
follows directly from Lemma~\ref{locally-isometric}.
\end{proof}

\begin{proposition} \label{mu-bounded} If a graph $G$ is $(s,s-1)$-dismantlable and the metric triangles of $G$ have sides of length at most $\mu$, then the intervals of $G$ are $(4s+2\mu)$-thin and $G$ is $(64s+36\mu)$-hyperbolic. If, additionally, $G$ is weakly modular, then the intervals of $G$ are $(4s+\mu)$-thin and $G$ is $(64s+20\mu)$-hyperbolic.
\end{proposition}

\begin{figure}[t]
\centering\includegraphics[scale=0.85]{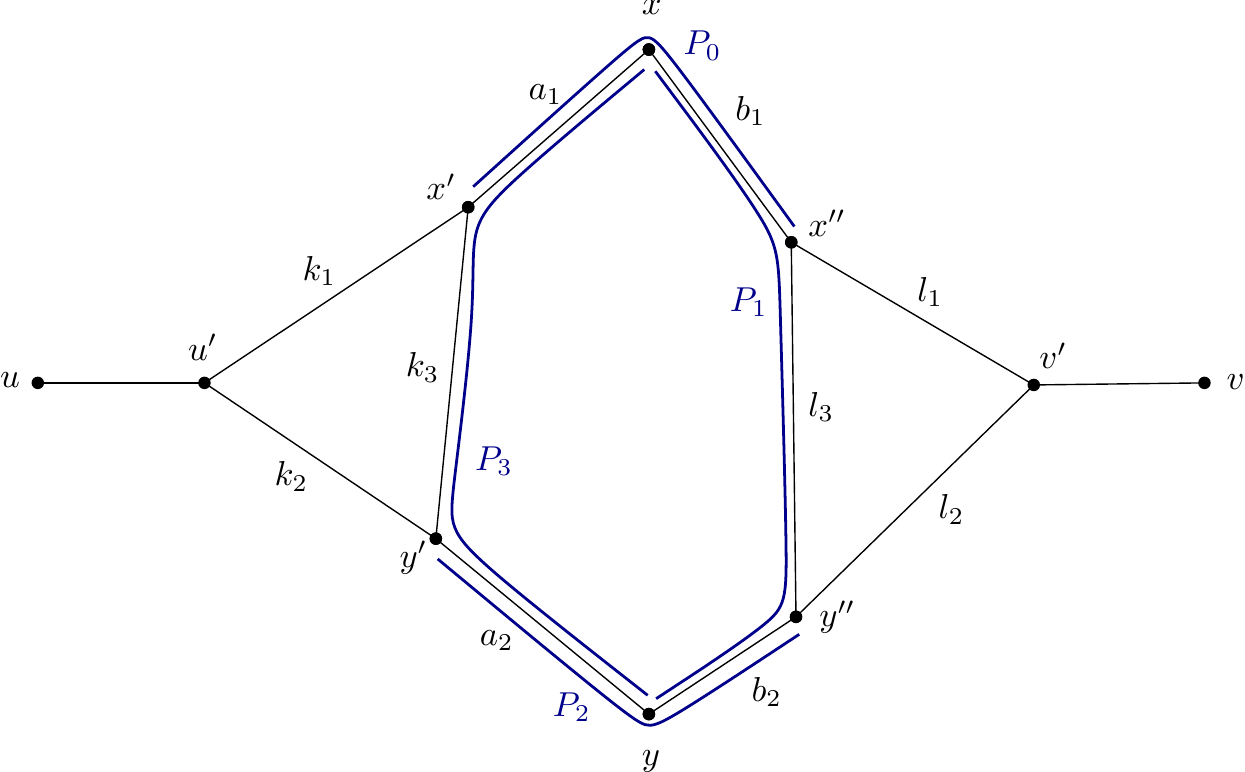}
\caption{To the proof of Proposition \ref{mu-bounded}}
\label{fig-bounded-interval}
\end{figure}

\begin{proof} Let $I(u,v)$ be an interval of $G$ and $x,y\in I(u,v)$
  such that $d(u,x)=d(u,y)=k, d(v,x)=d(v,y)=l$ and $l+k=d(u,v).$ Let
  $u'x'y'$ be a quasi-median of the triplet $u,x,y$ and let $v'x''y''$
  be a quasi-median of the triplet $v,x,y$ (see
  Figure~\ref{fig-bounded-interval}). Let $k_1=d(u',x'), k_2=d(u',y'),
  k_3=d(x',y'),$ and $a_1=k-k_1-d(u,u'), a_2=k-k_2-d(u,u').$
  Analogously, let $l_1=d(v',x''), l_2=d(v',y''), l_3=d(x'',y''),$ and
  $b_1=l-l_1-d(v,v'), b_2=l-l_2-d(v,v').$ Since the metric triangles
  of $G$ are $\mu$-bounded, each of $k_1,k_2,k_3,l_1,l_2,l_3$ is at
  most $\mu$.

Now, suppose that each of $a_1,a_2,b_1,b_2$ is at least $2s$. Let $c$
be a cycle consisting of a geodesic $R_1$ between $x$ and $x',$
followed by a geodesic $R_2$ between $x'$ and $y'$, a geodesic $R_3$
between $y',y$, a geodesic $Q_3$ between $y$ and $y'',$ a geodesic
$Q_2$ between $y''$ and $x'',$ and finally, a geodesic $Q_1$ between
$x''$ and $x$.  From the definition of quasi-medians it follows that
$P_1:=R_1\cup R_2\cup R_3$ and $P_3:=Q_1\cup Q_2\cup Q_3$ are
geodesics between $x$ and $y$. Since $x\in I(u,v)$, $x'\in I(u,x)$ and
$x''\in I(x,v)$, we also conclude that $P_0:=R_1\cup Q_1$ is a
geodesic between $x'$ and $x''$. Analogously, $P_2:=R_3\cup Q_3$ is a
geodesic between $y'$ and $y''$.  Since each pair of (circularly)
consecutive paths $P_0,P_1,P_2,P_3$ intersect along a path of length
at least $2s$ and any two nonconsecutive paths do not share common
edges, the set $P_0,P_1,P_2,P_3$ constitutes an $s$-geodesic covering
of $c$, but this is impossible by Lemma \ref{no-cycles}. This
contradiction shows that $\min\{ a_1,a_2,b_1,b_2\}<2s.$

Suppose without loss of generality that $a_1<2s.$ Since $k_1+a_1=k_2+a_2$ and $k_1,k_2\le \mu$, we obtain that $a_2\le a_1+k_1\le 2s+\mu.$ Since $x',y'$ belong to a common geodesic between $x$ and $y,$ we obtain that $d(x,y)=d(x,x')+d(x',y')+d(y',y)=a_1+k_3+a_2\le 2s+\mu+2s+\mu=4s+2\mu.$ Thus the intervals of $G$ are $(4s+2\mu)$-thin. Proposition \ref{mu-nu} shows that $G$ is $(64s+36\mu)$-hyperbolic.

If $G$ is weakly modular, then all metric triangles of $G$ are equilateral, whence $k_1=k_2=k_3$ and $a_1=a_2.$ This shows that $d(x,y)=a_1+k_3+a_2\le 4s+\mu$. Hence the intervals of $G$ are $(4s+\mu)$-thin and $G$ is $(64s+20\mu)$-hyperbolic.
\end{proof}

Next we will prove that the metric triangles of $(s,s-1)$-dismantlable weakly modular graphs are $6s$-bounded.

\begin{lemma} \label{triangles} Let $uvw$ be a metric triangle of a weakly modular graph $G$. For any vertex $x\in I(u,w)$ at distance $p$ from $u$ there exists a vertex $y\in I(u,v)$
at distance $p$ from $u$ and $x$.
\end{lemma}

\begin{proof} Let $u'v'x'$ be a quasi-median of the triplet $u,v,x$. Since $uvw$ is a metric triangle, $I(u,v)\cap I(u,w)=\{ u\}.$ Since $I(u,x)\subseteq I(u,w),$ necessarily $I(u,v)\cap I(u,x)=\{ u\},$ i.e., $u'=u.$ We also claim that $x=x'$. Since $x'\in I(u,x)\cap I(x,v),$ if $x\ne x'$, two different vertices $x$ and $x'$ of $I(u,w)$ will have different distances from $v$, contrary to Lemma \ref{lem-weakly-modular}. So, $x'=x$. Since $v'ux$ is an equilateral metric triangle, $d(v',x)=d(v',u)=d(u,x)=p$ and we are done.
\end{proof}

\begin{proposition} \label{weakly-modular-bounded} If $G$ is an
  $(s,s-1)$-dismantlable weakly modular graph, then the metric
  triangles of $G$ are $6s$-bounded.
\end{proposition}

\begin{proof} Suppose by way of contradiction that $G$ contains a
  metric triangle $uvw$ with sides of length $\ge 6s$. Let $x_u$ be a
  vertex of $I(u,w)$ located at distance $2s$ from $u$. Let $y_u$ be a
  vertex of $I(u,v)$ located at distance $2s$ from $u$ and $x_u$ (such
  a vertex exists by Lemma \ref{triangles}). Let $x_v$ be a vertex of
  $I(v,y_u)$ at distance $2s$ from $v$ and let $y_v$ be a vertex of
  $I(v,w)$ at distance $2s$ from $x_v$ and $v$ (again, this vertex is
  provided by Lemma \ref{triangles}). Finally, let $x_wy_ww$ be a
  quasi-median of the triplet $y_v,x_u,w$. Denote by $k$ the length of
  the sides of the metric triangle $x_wy_ww$. We distinguish two cases
  depending of the value of $k$.

\begin{figure}[t]
\centering\includegraphics[scale=0.85]{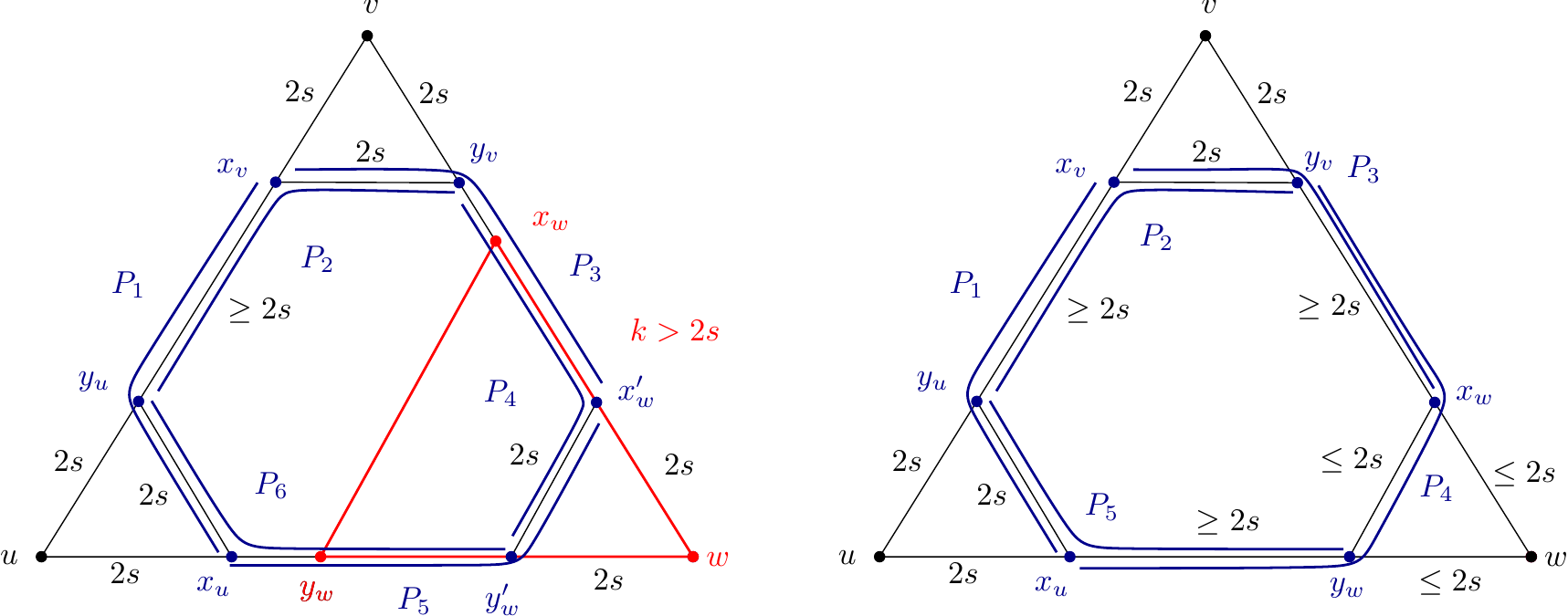}
\caption{Cases 1 and 2 of the proof of Proposition \ref{weakly-modular-bounded}}
\label{figure1}
\end{figure}

\medskip\noindent
{\bf Case 1:} $k>2s$ (Fig.~\ref{figure1} left).

\medskip
Let $x'_w$ be a vertex of $I(w,x_w)$ at distance $2s$ from $w$ and let
$y'_w$ be a vertex of $I(w,y_w)$ located at distance $2s$ from $w$ and
$x'_w$ (provided by Lemma \ref{triangles}). Denote by $P_1$ a geodesic
of $G$ between $x_u$ and $x_v$ passing via $y_u$ (such a geodesic
exists because $y_u\in I(u,v)\cap I(x_u,v)$ and $x_v\in
I(y_u,v)$). Let $P_2$ be a geodesic between $y_u$ and $y_v$ passing
via $x_v$, $P_3$ be a geodesic between $x_v$ and $x'_w$ passing via
$y_v,$ $P_4$ be a geodesic between $y_v$ and $y'_w$ passing via
$x'_w,$ $P_5$ be a geodesic between $x'_w$ and $x_u$ passing via
$y'_w$, and $P_6$ be a geodesic between $y'_w$ and $y_u$ passing via
$x_u$ (the existence of these geodesics follows from the way the
vertices $x_v,y_v,x'_w,y'_w,x_u$ have been selected and is similar to
the proof of existence of $P_1$). Let $c$ be the cycle of $G$ defined
as the union of these six geodesics. Since any two consecutive
geodesics intersect along a path of length at least $2s$ (because the
length of the sides of $uvw$ is at least $6s$ and the sides of the
metric triangles $ux_uy_u$, $x_vvy_v$, and $x'_wwy'_w$ have length
$2s$) and any two nonconsecutive geodesics are disjoint or intersect
in a single vertex, the cycle $c$ is $s$-geodesically covered by
$P_1,\ldots,P_6$, leading to a contradiction with Lemma
\ref{no-cycles}.

\medskip\noindent
{\bf Case 2:} $k\le 2s$ (Fig~\ref{figure1} right).

\medskip
In this case, we define the following five geodesics: $P_1$ is a geodesic between $x_u$ and $x_v$ passing via $y_u$, $P_2$ is a geodesic between $y_u$ and $y_v$ passing via $x_v$, $P_3$ is a geodesic between $x_v$ and $x_w$ passing via $y_v,$ $P_4$ is a geodesic between $y_v$ and $x_u$ passing via $x_w$ and $y_w$, and $P_5$ is a geodesic between $y_w$ and $y_u$ passing via $x_u$. The proof of existence of geodesics $P_1,P_2,P_3,$ and $P_5$ is the same as in Case 1. The existence of $P_4$ follows from the fact that $x_wy_ww$ is a quasi-median of the triplet $y_v,x_u,w$. Again, any two consecutive geodesics intersect along a path of length at least $2s$ while two nonconsecutive geodesics either are disjoint or intersect in a single vertex. Thus the cycle $c$, which is the union of these five geodesics, is $s$-geodesically covered by them, contrary to Lemma \ref{no-cycles}.

In both cases, the assumption that $G$ contains a metric triangle with
sides of length $\ge 6s$ leads us to a contradiction. Thus all metric
triangles of $G$ are $6s$-bounded.
\end{proof}

Propositions \ref{mu-bounded} and \ref{weakly-modular-bounded} imply that if a weakly modular graph $G$ is $(s,s-1)$-dismantlable, then $G$ is $184s$-hyperbolic. This completes the proof of Theorem \ref{weakly-modular}.
\end{proof}

Theorem \ref{weakly-modular} can be sharpened in case of median and
Helly graphs that are important subclasses of weakly-modular
graphs~\cite{BaCh_survey}. A graph $G=(V,E)$ is \emph{median} if for
all vertices $u,v,w \in V$, there exists a unique vertex $m$ in the
intersection $I(u,v) \cap I(u,w) \cap I(v,w)$. A graph $G$ is
\emph{Helly} if for any family $\cB$ of balls of $G$, the intersection
$\cap\{B \in \cB\}$ is non-empty if and only if $B \cap B' \neq
\emptyset$ for all $B, B' \in \cB$.

\begin{corollary} Let $G$ be an $(s,s-1)$-dismantlable graph. If $G$ is a median graph, then $G$ is $2s$-hyperbolic. If $G$ is a Helly graph, then $G$ is $(64s+20)$-hyperbolic.
\end{corollary}

\begin{proof}
It is known \cite{ChDrEsHaVa,Ha} that a median graph $G$ is
$\delta$-hyperbolic if and only if $G$ does not contain square
$\delta\times \delta$ grids as isometric subgraphs. To obtain our
result is suffices to show that if a graph $G$ is
$(s,s-1)$-dismantlable, then $G$ does not contain square grids of size
$2s\times 2s$ as isometric subgraphs.  Suppose by way of contradiction
that $G$ contains an isometric $2s\times 2s$ grid $H$; denote the
boundary cycle of $H$ by $c$. Let $u,v,w,x$ be the four corners of
$H$. Let $P_0$ be the $(u,w)$-geodesic of $H$ passing via $v,$ $P_1$
be the $(v,x)$-geodesic of $H$ passing via $w$, $P_2$ be the
$(w,u)$-geodesic of $H$ passing via $x,$ and $P_3$ be the
$(x,v)$-geodesic of $H$ passing via $u$. These four geodesics show
that $c$ is a $s$-geodesically covered cycle, contrary to Lemma
\ref{no-cycles}.  This establishes the first assertion.

Now, let $G$ be an $(s,s-1)$-dismantlable Helly graph. Then $G$ is
weakly modular \cite{BaCh_survey}. We assert that all metric triangles
of $G$ have sides of length at most 1. Indeed, if $uvw$ is a metric
triangle with sides of length $k>1,$ consider the following three
pairwise intersecting balls: $B_1(u), B_{k-1}(v),$ and $B_{k-1}(w).$
By Helly property, they have a common vertex $u'$. But then $u'\in
I(u,v)\cap I(u,w)$ and $u'\ne u$, because $d(u,v)=k$ and $u'\in
B_{k-1}(v),$ contrary to the assumption that $uvw$ is a metric
triangle. Hence $k\le 1.$ By Proposition \ref{mu-bounded}, $G$ is
$(64s+20)$-hyperbolic.
\end{proof}


\section{Algorithmic consequences}

The {\it hyperbolicity} $\delta^*$ of a metric space $(X,d)$ (or of a
non-necessarily finite graph $G$) is the least value of $\delta$ for
which $(X,d)$ (resp., $G$) is $\delta$-hyperbolic.  By a remark of
Gromov \cite{Gr}, if the four-point condition in the definition of
hyperbolicity holds for a fixed base-point $u$ and any triplet $x,y,v$
of $X$, then the metric space $(X,d)$ is $2\delta$-hyperbolic. This
provides a factor 2 approximation of hyperbolicity of a metric space
on $n$ points running in cubic $O(n^3)$ time. Using fast algorithms
for computing (max,min)-matrix products, it was noticed in
\cite{FouIsVi} that this 2-approximation of hyperbolicity can be
implemented in $O(n^{2.69})$ time. In the same paper, it is shown that
any algorithm computing the hyperbolicity for a fixed base-point in
time $O(n^{2.05})$ would provide an algorithm for $(\max,\min)$-matrix
multiplication faster than the existing ones.  In~\cite{Duan},
approximation algorithms are given to compute a
$(1+\epsilon)$-approximation in $O(\epsilon^{-1}n^{3.38})$ time and a
$(2+\epsilon)$-approximation in $O(\epsilon^{-1}n^{2.38})$ time.  For
a practical motivation of a fast computation or approximation of
hyperbolicity of large graphs and an experimental study, see
\cite{CoCoLa}.

Gromov gave an algorithm to recognize Cayley graphs of hyperbolic
groups and estimate the hyperbolicity constant $\delta $. His
algorithm is based on the theorem that hyperbolicity ``propagates'',
i.e. if balls of an appropriate fixed radius are hyperbolic for a
given $\delta $ then the whole space is $\delta '$-hyperbolic for some
$\delta '>\delta $ (see \cite{Gr}, 6.6.F). More precisely, for simply
connected\footnote{Recall that a topological space $X$ is \emph{simply
    connected} if it is path-connected (i.e., for all points $x,y \in
  X$, there exists a path from $x$ to $y$ in $X$) and every loop is
  null-homotopic (i.e., can be continuously deformed to a point).}
geodesic spaces, hyperbolicity can be characterized in the following
local-to-global way:

\begin{theorem}\cite{Gr}\label{Gromov}  Given $\delta >0$, let $R = 10^5\delta $ and $\delta'=200\delta $. Let  $(X,d)$ be a simply connected geodesic metric space in which each loop of length
$<100\delta$ is null-homotopic inside a ball of diameter $<200\delta$. If every ball $B_R(x_0)$ of $X$ is $\delta $-hyperbolic, then $X$
is $\delta'$-hyperbolic.
\end{theorem}

Although Cayley graphs (viewed as 1-dimensional complexes) are not
simply connected, they can be replaced by the (2-dimensional) Cayley
complexes of the groups, which are simply connected, and the theorem
above applies. To check the hyperbolicity of a Cayley graph it is
enough to verify the hyperbolicity of a sufficiently big ball (note
that all balls of a given radius in the Cayley graph are isomorphic to
each other). For other versions of this ``local-to-global'' theorem
for hyperbolicity see \cite {Bowd2}, \cite {DelGro}, \cite
{Pa2}. However this theorem does not help when dealing with arbitrary
graphs due to the simple-connectedness assumptions.

\subsection{Approximating the hyperbolicity of a graph}

In this section, we will describe a fast $O(n^2)$ time algorithm for
constant-factor approximation of hyperbolicity $\delta^*$ of a graph $G=(V,E)$ with $n$ vertices and $m$ edges,
assuming that its distance-matrix has already been
computed.  Our algorithm is very simple and can be used as a practical
heuristic to approximate the hyperbolicity of graphs.

The hyperbolicity $\delta^*$ of a graph $G$ is an integer or a
half-integer belonging to the list
$\{0,\frac{1}{2},1,\frac{3}{2},2,\ldots n-1,\frac{2n-1}{2},n\}$. It is
known that $0$-hyperbolic graphs are exactly the block graphs, i.e.,
the graphs in which every $2$-connected component is a
clique~\cite{BaMu}. Consequently, from the distance-matrix of $G$, one
can check in time $O(n^2)$ whether $\delta^* = 0$ or not. In the
following, we assume that $\delta^* \geq \frac{1}{2}$.


Before presenting the general algorithm (Algorithm~\ref{algo-approx}),
we describe an auxiliary algorithm
(Algorithm~\ref{algo-approx-alpha}) that for a parameter $\alpha$
either ensures that $G$ is $(784\alpha+\frac{1}{2})$-hyperbolic
or  that $G$ is not $\frac{\alpha}{2}$-hyperbolic. Algorithm~\ref{algo-approx-alpha} is
based on Theorem \ref{theorem-main} and Corollary
\ref{cor-delta-implies-copwin} of Proposition~\ref{prop-VB}.

\begin{algorithm}
\SetKwFor{Foreach}{for each}{do}{endfor}
\caption{Approximated-Hyperbolicity($G=(V,E),\alpha$)\label{algo-approx-alpha}}
Construct a BFS-order $\preceq$ starting from an arbitrary vertex
$v_0$\;

For each $v\in V$, let $f_\alpha(v)$ be the vertex at distance
$\min\{2\alpha,d(v,v_0)\}$ from $v$ on the path of the BFS-tree from
$v$ to $v_0$ \;

\Foreach{$v \in V$}
{\lIf{$B_{4\alpha}(v) \cap X_v \not\subseteq
  B_{3\alpha}(f_\alpha(v))$}{\Return \textsc{No}}}
\Return \textsc{Yes}
\end{algorithm}



First, suppose that Algorithm~\ref{algo-approx-alpha} returns
\textsc{Yes}. This means that the BFS-order $\preceq$ is a
$(4\alpha,3\alpha)^*$-dismantling order of the vertices of
$G$. Consequently, from Theorem~\ref{theorem-main}, $G$ is
$(784\alpha+\frac{1}{2})$-hyperbolic.  Now, suppose that the algorithm
returns \textsc{No}. This means that there exists a vertex $v$ such
that $B_{4\alpha}(v) \cap X_v \not\subseteq
B_{3\alpha}(f_\alpha(v))$. From Proposition~\ref{prop-VB} with
$r=2\alpha$, this implies that $G$ is not $\frac{\alpha}{2}$-hyperbolic
and thus $\delta^* > \frac{\alpha}{2}$.

Algorithm~\ref{algo-approx} efficiently computes the smallest
integer $\alpha$ for which the Algorithm~\ref{algo-approx-alpha} returns the answer
\textsc{Yes}, i.e, the smallest integer $\alpha$ for which the inclusion $B_{4\alpha}(v) \cap X_v \subseteq
  B_{3\alpha}(f_\alpha(v))$ holds for all vertices $v$ of $G$.  Similarly to
Algorithm~\ref{algo-approx-alpha}, we assume that we have constructed a BFS-order $\preceq$ of
the vertices of $G$ starting from an arbitrary but fixed vertex $v_0$. Suppose
that for each vertex $v$, $p(v)$ denotes  the parent of $v$ in the BFS-tree
corresponding to $\preceq$ (with the convention that $p(v_0) =
v_0$). As in Algorithm~\ref{algo-approx-alpha}, for each vertex $v$ and for each value of
$\alpha$, let $f_\alpha(v)$ be the vertex at distance
$\min\{2\alpha,d(v,v_0)\}$ from $v$ located on the path of the
BFS-tree from $v$ to $v_0$. Note that $f_{\alpha+1}(v) = p(p(f_\alpha(v)))$.

We start  with a lemma ensuring that during the execution of the
algorithm, we do not have to completely recompute the balls
$B_{3\alpha}(f_\alpha(v))$ each time we modify $\alpha$.

\begin{lemma}\label{lem-monotone}
If $\alpha' \leq \alpha$, then
$B_{3\alpha'}(f_{\alpha'}(v)) \subseteq B_{3\alpha}(f_\alpha(v))$ for any vertex $v$ of $G$.
\end{lemma}

\begin{proof}
It is enough to prove that for any $\alpha\ge 0$ and any $v \in V$,
$B_{3\alpha}(f_{\alpha}(v)) \subseteq
B_{3(\alpha+1)}(f_{\alpha+1}(v))$.  Pick any vertex $v \in V$. Since
 $f_{\alpha+1}(v) = p(p(f_\alpha(v)))$, we have
$d(f_{\alpha+1}(v),f_\alpha(v)) \leq 2$. Let $w \in
B_{3\alpha}(f_{\alpha}(v))$. Note that $d(w,f_{\alpha+1}(v)) \leq
d(w,f_\alpha(v)) + d(f_\alpha(v),f_{\alpha+1}(v)) \leq 3\alpha + 2
\leq 3(\alpha+1)$. Consequently, $B_{3\alpha}(f_{\alpha}(v)) \subseteq
B_{3(\alpha+1)}(f_{\alpha+1}(v))$.
\end{proof}


\begin{algorithm}
\caption{Approximated-Hyperbolicity($G=(V,E)$)\label{algo-approx}}
\SetKwFunction{Done}{done}
\SetKwFunction{Push}{push}
\SetKwFunction{Pop}{pop}
\SetKwData{myTrue}{true}
\SetKwData{myFalse}{false}
\SetKwFor{Foreach}{for each}{do}{endfor}
Construct a BFS-order $\preceq$ starting from an arbitrary vertex $v_0$\;
For each $v \in V$, let $p(v)$  be the parent of $v$ in the BFS-tree \;
For each $v \in V$, let $L(v)$ be a stack containing all vertices of
$G$ sorted (increasingly) by their distance to $v$ \;

$\Done \leftarrow \myFalse$\;
$\alpha \leftarrow 0$\;
\lForeach{ $v \in V$}{$f_\alpha(v) \leftarrow v$}
\While{\emph{not} $\Done$}
{
  $\Done \leftarrow \myTrue$\;
  $\alpha \leftarrow \alpha + 1$\;
  \Foreach{$v \in V$}
      {
        $f_\alpha(v) \leftarrow p(p(f_\alpha(v)))$\;
        \Repeat{$d(u,v)>4\alpha$ \emph{or} 
          $(u \preceq v$ \emph{and} $d(u,f_\alpha(v)) >
          3\alpha)$}
               {$u \leftarrow \Pop(L(v))$} 
        \lIf{$d(u,v) \leq 4 \alpha$}{$\Done \leftarrow \myFalse$}
        $\Push(u,L(v))$\;
      }
}
\Return $\alpha$
\end{algorithm}


Algorithm~\ref{algo-approx} can be viewed as a ``sieve of $n$ stacks"
and works as follows. In the preprocessing step, for each vertex $v$
of $G$, we sort the vertices of $G$ according to their distances to
$v$ and successively insert them in a stack $L(v)$ (so that $v$ is the
head of $L(v)$).  Starting with $\alpha = 1$, for each vertex $v$ of
$G$, we compute $f_{\alpha}(v)$ and as long as the current head $u$ of
$L(v)$ is in $B_{4\alpha}(v)$ and is such that $v \preceq u$ or
$d(u,f_\alpha(v)) \leq 3\alpha$, we pop $u$ from $L(v)$.  The idea is
that none of those popped elements can be a witness for
$B_{4\alpha}(v)\cap X_v \not \subseteq B_{3\alpha}(f_\alpha(v))$. If
there exists a vertex $v$ which is at distance at most $4\alpha$ from
the head $u$ of its stack $L(v)$, then we have found a witness showing
that $B_{4\alpha}(v)\cap X_v \not \subseteq
B_{3\alpha}(f_\alpha(v))$. In this case, by Proposition~\ref{prop-VB},
we know that $G$ is not $\frac{\alpha}{2}$-hyperbolic. Thus, we
increment $\alpha$ by 1 and start a new iteration. Otherwise, if each
$v$ is at distance $>4\alpha$ from the current head of $L(v)$, then
Algorithm~\ref{algo-approx} returns the current $\alpha$ as the least
value for which the Algorithm~\ref{algo-approx-alpha} returns the
answer \textsc{Yes}.

\begin{proposition}\label{prop-algo-approx}
There exists a constant-factor approximation algorithm to approximate the hyperbolicity $\delta^*$ of a graph
$G$ with $n$ vertices running in $O(n^2)$ time if $G$ is given by its distance-matrix. The algorithm
returns a $1569$-approximation of $\delta^*$.
\end{proposition}

\begin{proof} We start with the correctness proof of Algorithm~\ref{algo-approx}
(the correctness
of Algorithm~\ref{algo-approx-alpha} was given above).
Suppose that we are at iteration $\alpha$ and consider an arbitrary vertex
$v$ of $G$.  Any vertex $w$ that has been removed from $L(v)$ at a previous
iteration $\alpha' < \alpha$ either satisfies $v \preceq w$ or
$d(w,f_{\alpha'}(v)) \leq 3\alpha'$. Consequently, by
Lemma~\ref{lem-monotone}, $v \preceq w$ or $d(w,f_{\alpha}(v)) \leq
3\alpha$. Therefore, at the end of iteration $\alpha$, all
vertices that have been already removed from $L(v)$  cannot serve as witnesses for
$B_{4\alpha}(v)\cap X_v\not \subseteq B_{3\alpha}(f_\alpha(v))$.

Suppose now that at the end of iteration $\alpha$, for every vertex $v$ of $G$, the
head $u$ of $L(v)$ satisfies the inequality $d(u,v) > 4\alpha$. Since
initially, all vertices of $G$ were inserted in $L(v)$ according to their distances
to $v$, this means that all vertices of
$B_{4\alpha}(v)$ have been removed from $L(v)$. Since each $w$
removed from $L(v)$  either appears after $v$ in $\preceq$ or has
distance at most $3\alpha$ from $f_\alpha(v)$, we conclude that
$B_{4\alpha}(v) \cap X_v \subseteq B_{3\alpha(v)}$. Consequently,
$\preceq$ is a $(4\alpha,3\alpha)^*$-dismantling order and by
Theorem~\ref{theorem-main}, $G$ is $(784\alpha+\frac{1}{2})$-hyperbolic.
Since we also know that $G$ is not $\frac{\alpha-1}{2}$-hyperbolic,
$\delta^* \leq 784\alpha + \frac{1}{2} \leq 1568\delta^* +\frac{1}{2}
\leq 1569\delta^*$.  This gives a $1569$-approximation of the
hyperbolicity $\delta^*$ of $G$.

As to the complexity, first note that computing the BFS-order
$\preceq$ and the value of $p(v)$ for each $v \in V$ can be done in
time $O(n^2)$ from the distance-matrix of $G$ (this can be done in
time linear in the number of edges of $G$ if we are also given the
adjacency list of $G$).  Since $|V|=n$ and all the pairwise distances
are integers between $0$ and $n$, one can construct each stack $L(v)$
in time $O(n)$ using a counting sort algorithm. Thus, the
preprocessing step requires total $O(n^2)$ time. Since during the
execution of the algorithm we always have $\alpha \leq 2\delta^* \leq
2n$, $\alpha$ is incremented at most $2n$ times. Since for each $v \in
V$, once a vertex $w$ is popped from $L(v)$, $w$ is no longer used for
$v$ at subsequent iterations, there are at most $O(n^2)$ \texttt{pop}
operations. Therefore Algorithm~\ref{algo-approx} terminates in time
$O(n^2)$.
\end{proof}

Now suppose that the input graph $G=(V,E)$ is given by the adjacency list
(instead  of its distance-matrix). In the most naive implementation of
Algorithm~\ref{algo-approx}, one can perform a BFS-traversal
of $G$ from each vertex of $V$ to compute the distance matrix of $G$ in time
$O(mn)$, where $m = |E|$ and $n=|V|$. One can also use Seidel's
algorithm~\cite{Seidel} to compute the distance-matrix of $G$ in time
$O(n^{2.38})$. Hence, we get immediately the following corollary.

\begin{corollary}
There exists a constant-factor approximation algorithm to approximate
the hyperbolicity $\delta^*$ of a graph $G$ with $n$ vertices and $m$
edges running in $O(\min (mn, n^{2.38}))$ time if
$G$ is given by its adjacency list.
\end{corollary}

However, note that once we have computed the BFS-order $\preceq$ in
time $O(m)$, all the remaining computations are local, around each
vertex. Namely, one can replace the preprocessing step of
Algorithm~\ref{algo-approx} by the following localized
computations. First, we modify slightly the algorithm such that each
time we increase $\alpha$, we consider the vertices of $V$ in the
order $\preceq$. 

Then, instead of having $L(v)$ as a stack we consider $L(v)$ as a
queue; initially each $L(v)$ contains a single vertex $v$, which is
labeled.  At iteration $\alpha$, when the head $u$ of $L(v)$ is
removed from the queue $L(v)$, we label and insert in $L(v)$ all still
unlabeled neighbors of $u$. These inserted vertices are one step
further from $v$ than $u$.  For each vertex $v$, the order in which
vertices are added in $L(v)$ corresponds to a BFS-order computed from
$v$.  Moreover, each time we insert a vertex $u$ in $L(v)$, we store
the distance $d(u,v)$.  In order to efficiently retrieve the computed
distances, we will use a matrix $D$. Initially, $D(u,v) = \infty$ for
all $u,v$ and when $u$ is inserted in $L(v)$, we update $D(u,v)$ to
$d(u,v)$. Using a standard trick (see~\cite[ex. 2.12]{AHU}), one can
avoid the initialization cost of the matrix $D$. 

All the remaining steps of the Algorithm~\ref{algo-approx} remain the
same.  When in Algorithm~\ref{algo-approx} we need the value of
$d(u,f_\alpha(v))$ where $u$ is the head of $L(v)$, we now use
$D(u,f_\alpha(v))$ instead. In order to prove that the algorithm is
still correct, it is enough to show that if $u\preceq v$, $d(u,v) \leq
4\alpha$, and $D(u,f_\alpha(v)) > 3\alpha$, then $G$ is not
$\frac{\alpha}{2}$-hyperbolic and thus we need to increment
$\alpha$. We prove this property by induction on $\preceq$. Note that
if $D(u,f_\alpha(v)) \neq \infty$, then $D(u,f_\alpha(v)) =
d(u,f_\alpha(v))$ and $u \in (B_{4\alpha}(v) \cap X_{v})\setminus
B_{3\alpha}(f_\alpha(v))$. Thus, by Proposition~\ref{prop-VB}, $G$ is
not $\frac{\alpha}{2}$-hyperbolic. In particular, this is the case if
$v = v_0$, since $f_\alpha(v_0) = v_0$ and $D(u,f_\alpha(v_0)) =
d(u,v_0)$. Suppose now that $D(u,f_\alpha(v)) = \infty$ and let $v' =
f_\alpha(v)$. Let $u'$ be the head of $L(v')$ and note that
$D(u',v')=d(u',v')$. Since $u$ has not yet been added to $L(v')$,
necessarily, $d(u,v') \geq d(u',v')$. Thus, if $D(u',v') > 4\alpha$,
we have that $d(u,v') > 4\alpha > 3\alpha$ and by
Proposition~\ref{prop-VB}, $G$ is not $\frac{\alpha}{2}$-hyperbolic
since $u \in (B_{4\alpha}(v) \cap X_{v})\setminus
B_{3\alpha}(f_\alpha(v))$. Suppose now that $D(u',v') \leq
4\alpha$. Since $v' = f_\alpha(v) \preceq v$, we have already iterated
over $v'$ at step $\alpha$, and consequently, $u' \preceq v'$ and
$D(u',f_\alpha(v')) > 3\alpha$. Thus, by induction hypothesis, we know
that $G$ is not $\frac{\alpha}{2}$-hyperbolic.

Since $\alpha$ is never greater than $2\delta^*$, with this
implementation, the complexity of Algorithm~\ref{algo-approx} becomes
$O(\sum_{v\in V}|E(B_{8\delta^*+1}(v))|),$ where $|E(B_{8\delta^*+1}(v))|$
is the number of edges in the subgraph of $G$ induced by the ball
$B_{8\delta^*+1}(v)$.  This is efficient if the balls $B_{8\delta^*+1}(v)$
do not contain too many vertices and edges, in particular if $G$ is a
bounded-degree graph of small hyperbolicity. Consequently, we obtain
the following result.

\begin{proposition}
If $G$ is a graph with $n$ vertices and $m$ edges, given by its adjacency
list, then  Algorithm~\ref{algo-approx} can be implemented to run  in
$O(\sum_{v\in V}|E(B_{8\delta^*+1}(v)|)$ time. In particular, if there 
exists a constant $K>0$ such that for each $v \in V$
the ball $B_{8\delta^*+1}(v)$ contains at most $K$ edges, then its complexity becomes
$O(Kn)$. 
\end{proposition}

Finally, in the case of weakly modular graphs, we can obtain a sharper
approximation of hyperbolicity using the same ideas. The following
result is the counterpart of Proposition~\ref{prop-algo-approx} for
weakly modular graphs.

\begin{proposition}
If $G$ is a weakly modular graph given by its distance matrix, then
in time $O(n^2)$ one can compute $\delta'$  such that $\delta^* \leq
\delta' \leq 736\delta^*+368$.
\end{proposition}

\begin{proof}
Consider a weakly modular graph graph $G=(V,E)$ and assume that we have constructed a BFS-order
$\preceq$ on $V$ starting from an arbitrary vertex $v_0$; as before, assume that
for each $v$, $p(v)$ is the parent of $v$ in the corresponding BFS-tree
(with $p(v_0) = v_0$). For each vertex $v$ and each integer $\alpha$,
let $g_\alpha(v)$ be the vertex at distance
$\min\{\alpha+1,d(v,v_0)\}$ on the path of the BFS-tree from $v$ to
$v_0$. Note that $g_{\alpha+1}(v) = p(g(\alpha(v))$.

We want to find the smallest $\alpha$ such that $B_{2\alpha+2}(v)
\cap X_v \subseteq B_{2\alpha+1}(g_\alpha(v))$ for all vertices
$v$. We first note that for any vertex $v$ and any value of $\alpha$,
$B_{2\alpha+1}(g_\alpha(v)) \subseteq
B_{2(\alpha+1)+1}(g_{\alpha+1}(v))$. Indeed, since
$d(g_\alpha(v),g_{\alpha+1}(v)) \leq 1$, for every $u \in
B_{2\alpha+1}(g_\alpha(v))$, by triangle inequality $d(u,
g_{\alpha+1}(v)) \leq 2\alpha + 1 +1 \leq
2(\alpha+1)+1$. Consequently, one only need to slightly adapt
Algorithm~\ref{algo-approx} to compute the smallest $\alpha$ such that
$B_{2\alpha+2}(v) \cap X_v \subseteq B_{2\alpha+1}(g_\alpha(v))$
for all vertices $v$. This algorithm runs in time $O(n^2)$.

Suppose now that we have computed the smallest integer $\alpha$ such
that $B_{2\alpha+2}(v) \cap X_v \subseteq
B_{2\alpha+1}(g_\alpha(v))$ for all vertices $v$. By
Theorem~\ref{weakly-modular}, $G$ is
$184(2\alpha+2)$-hyperbolic. Moreover, since $B_{2\alpha}(v) \cap
X_v \not\subseteq B_{2\alpha-1}(g_{\alpha-1}(v))$, by
Proposition~\ref{prop-VB} with $r=\alpha$, we have that $G$ is not
$\frac{\alpha-1}{2}$-hyperbolic and thus $\delta^* \geq
\frac{\alpha}{2}$. Consequently, $\frac{\alpha}{2} \leq \delta^* \leq
368(\alpha + 1)$ and thus, $\delta^* \leq 368(\alpha +1) \leq 368
(2\delta^* +1)$.
\end{proof}




\subsection{Graphs with balanced metric triangles}

We conclude our paper with another local-to-global condition for
hyperbolicity, analogous to Theorem \ref{Gromov}. Namely, we replace
the topological condition of simple connectivity by a metric
condition (this result can be combined with algorithms from previous subsection
to estimate the hyperbolicity of a graph).

Given a strictly increasing function $f : \N \rightarrow \N$, a graph
$G$ has \emph{$f$-balanced metric triangles} if for every metric triangle
$uvw$, $d(u,v) \leq \min \{f(d(u,w)), f(d(v,w))\}$.  In other words,
for any metric triangle $uvw$, if one side of the triangle is
``small'', then the other two sides are ``relatively small'' too.  When $f(k) =
C\cdot k$ for some constant $C$, we say that the metric triangles of
$G$ are \emph{linearly balanced}. If $G$ is a weakly modular graph,
then $G$ has linearly balanced triangles; indeed, all metric triangles
of $G$ are equilateral, thus one can choose $f(k) = k$.





\begin{proposition}\label{prop:lambda-equilateral}
Let $G$ be a graph with $f$-balanced metric triangles. If every
ball $B_{f(12\delta)+8\delta}(v)$ of $G$ is $\delta$-hyperbolic with
$\delta > 0$, then $G$ is $1569\delta$-hyperbolic. Moreover, if $G$ is
a weakly modular graph such that every ball $B_{10\delta+5}(v)$ of
$G$ is $\delta$-hyperbolic, then $G$ is $(736\delta+368)$-hyperbolic.
\end{proposition}

\begin{proof}
Consider a graph $G$ with $f$-balanced metric triangles where every
ball $B_{f(12\delta)+8\delta}(v)$ is $\delta$-hyperbolic and assume
that $G$ is not $1569\delta$-hyperbolic. From
Theorem~\ref{theorem-main}, it implies that $G$ is not
$(4\alpha,3\alpha)^*$-dismantlable for $\alpha = 2\delta$. Let $z$ be
an arbitrary vertex of $G$ and consider a BFS-order $\preceq$ of $G$
rooted at $z$. Since $G$ is not $(4\alpha,3\alpha)^*$-dismantlable,
there exist $x,y,c$ such that $c \in I(z,x)$, $d(c,x) = 2\alpha$,
$d(z,y) \leq d(z,x)$, $d(x,y) \leq 4\alpha$, and $d(c,y) >
3\alpha$.

Let $z'c'y'$ be a quasi-median of the triplet $z,c,y$. Since $z'\in
I(z,c)\subset I(z,x),$ $z'\in I(z,y),$ and $d(y,z)\le d(x,z),$
necessarily $d(y,z') \leq d(x,z')$.  Moreover, since $y' \in I(c,y)$,
$d(c,y') \leq d(c,y) \leq d(c,x)+d(x,y) \leq 6\alpha$.  Since the
metric triangles of $G$ are $f$-balanced, $d(c',z') \leq
f(d(c',y'))$. Since $f: \N \rightarrow \N$ is a strictly increasing
function, $d(c,z') = d(c,c')+d(c',z') \leq d(c,c') +f(d(c',y')) \leq
f(d(c,c')+d(c',y')) = f(d(c,y')) = f(6\alpha)$.  Consequently,
$d(y,z')\leq d(x,z') \leq d(x,c)+d(c,z') \leq
2\alpha+f(6\alpha)$. Since $d(x,y) \leq 4\alpha$, for every vertex $u
\in I(x,y)$, $d(u,z') \leq \min\{ d(u,x)+d(x,z'),d(u,y)+d(y,z')\}\le
2\alpha+2\alpha +f(6\alpha)=4\alpha+ f(6\alpha)$. Consequently, the
distance between $x$ and $y$ in the graph $G$ and in the ball $B :=
B_{f(6\alpha)+4\alpha}(z')$ are the same, thus $d_B(x,y) \leq
4\alpha$.  Note that, since $c \in I(x,z)$ and $z' \in I(c,z)$, $c \in
I(x,z')$. Since $I(x,z') \subseteq B$ and $I(y,z') \subseteq B$,
$d_B(c,x) = 2\alpha$ and $d_B(y,z') = d(y,z') \leq d(x,z') =
d_B(x,z')$.  Since $d_B(y,z') \leq d_B(x,z')$, $d_B(c,x) = 2\alpha$,
$d_B(x,y) \leq 4\alpha$, and $d_B(c,y) \geq d(c,y)>3\alpha$, from
Proposition~\ref{prop-VB} applied with $r = 2\alpha$,
$B_{f(6\alpha)+4\alpha}(z')$ is not
$\frac{\alpha}{2}$-hyperbolic. Thus, since $\alpha = 2\delta$, the
ball $B_{f(12\delta)+8\delta}(z')$ is not $\delta$-hyperbolic, which
is a contradiction.

Consider now a weakly modular graph $G$ where every ball
$B_{10\delta+5}(v)$ is $\delta$-hyperbolic and assume that $G$ is
not $(736\delta+368)$-hyperbolic.  From Theorem~\ref{weakly-modular},
it implies that $G$ is not $(2\alpha+2,2\alpha+1)^*$-dismantlable for
$\alpha = 2\delta$. Thus, there exist $x,y,z,c$ such that $c \in
I(z,x)$, $d(c,x) = \alpha+1$, $d(z,y) \leq d(z,x)$, $d(x,y) \leq
2\alpha+2$ and $d(c,y) > 2\alpha+1$. Let $z'c'y'$ be a quasi-median of
the triplet $z,c,y$. Using the same arguments as in the previous case
and the fact that metric triangles of weakly modular graphs are
equilateral, one can show that the ball $B_{5\alpha+5}(z')$ is not
$\frac{\alpha}{2}$-hyperbolic. Consequently the ball
$B_{10\delta+5}(z')$ is not $\delta$-hyperbolic, which is a
contradiction.
\end{proof}

\section*{Acknowledgements}

 We wish to thank the anonymous referees for a careful reading of the
 manuscript, their useful suggestions, and for pointing out the result
 of Soto~\cite{Soto-PhD}.

 J.C. was partially supported by ANR project MACARON
 (\textsc{anr-13-js02-0002}).  V.C. was partially supported by ANR
 projects TEOMATRO (\textsc{anr-10-blan-0207}) and GGAA
 (\textsc{anr-10-blan-0116}).

\end{document}